\pdfoutput=1

\documentclass[options]{article}
\usepackage{amsmath}
\usepackage{amsfonts}
\usepackage{amssymb}
\usepackage{amsthm}
\usepackage{mathtools}
\usepackage{titling}
\usepackage{authblk}
\usepackage{commath}
\usepackage{xcolor}
\usepackage{graphicx}
\usepackage{subcaption}
\usepackage{bbm}
\usepackage{cite,hyperref}
\usepackage{color,soul}
\usepackage{bm}
\usepackage{mathrsfs}

\usepackage{enumitem}
\newtheorem{remark}{Remark}[section]
\newtheorem{theorem}[remark]{Theorem}
\newtheorem{corollary}[remark]{Corollary}
\newtheorem{lemma}[remark]{Lemma}
\usepackage[margin=1in]{geometry}

\newtheorem{proposition}[remark]{Proposition}

\numberwithin{equation}{section}
\allowdisplaybreaks

\newcommand{\nocontentsline}[3]{}
\newcommand{\tocless}[2]{\bgroup\let\addcontentsline=\nocontentsline#1{#2}\egroup}
\usepackage{tocloft}
\usepackage{blindtext}
 \cftpagenumbersoff{section}
 \cftpagenumbersoff{subsection}
 \cftpagenumbersoff{subsubsection}
\makeatletter \renewcommand{\@cftmaketoctitle}{} \makeatother
\usepackage{bookmark}

\DeclareMathOperator*{\esssup}{ess\,sup}

\title{On the Obstacle Problem in Fractional Generalised Orlicz Spaces}
\author{Catharine W.K. Lo\thanks{Department of Mathematics, City University of Hong Kong, Hong Kong SAR, China\\ Email address: wingkclo@cityu.edu.hk} \, and Jos\'e Francisco Rodrigues\thanks{CMAFcIO -- Departamento de Matem\'atica, Faculdade de Ci\^encias, Universidade de Lisboa P-1749-016 Lisboa, Portugal\\ Email address: jfrodrigues@ciencias.ulisboa.pt}}
\date{}

\begin{document}
\maketitle

\begin{abstract}
    We consider the one and the two obstacles problems for the nonlocal nonlinear anisotropic $g$-Laplacian $\mathcal{L}_g^s$, with $0<s<1$. We prove the strict T-monotonicity of $\mathcal{L}_g^s$  and we obtain the Lewy-Stampacchia inequalities. We consider the approximation of the solutions through semilinear problems, for which we prove a global $L^\infty$-estimate, and we extend the local Hölder regularity to the solutions of the obstacle problems in the case of the fractional $p(x,y)$-Laplacian operator. We make further remarks on a few elementary properties of related capacities in the fractional generalised Orlicz framework, with a special reference to the Hilbertian nonlinear case in fractional Sobolev spaces.
\end{abstract}

\tocless\section{Introduction}
\hypertarget{introduction}{}
\bookmark[level=section,dest=introduction]{Introduction}

In this work, we consider nonlocal nonlinear anisotropic operators of the $g$-Laplacian type $\mathcal{L}_g^s:W^{s,G_{:}}_0(\Omega)\to W^{-s,G_{:}^*}(\Omega)$, in Lipschitz bounded domains $\Omega\subset\mathbb{R}^d$, as defined in \cite{ByunKimOk2022MathAnnLocalHolderFracGLapAnisotropic,ChakerKimWeidner2022MathAnnHarnackFracAnisotropicGLap,ChakerKimWeidner2022CVPDERegGLap} by 
\begin{equation}\label{Lgdef}\langle\mathcal{L}_g^su,v\rangle=\int_{\mathbb{R}^d}\int_{\mathbb{R}^d}g\left(x,y,|\delta^s u(x,y)|\right)\delta^s u(x,y)\delta^s v(x,y)\frac{dx\,dy}{|x-y|^d},\end{equation} 
where $\langle\cdot,\cdot\rangle$ denotes the duality between $W^{s,G_{:}}_0(\Omega)$ and its dual space $W^{-s,G_{:}^*}(\Omega)=[W^{s,G_{:}}_0(\Omega)]^*$, for the fractional generalised Orlicz space $W^{s,G_{:}}_0(\Omega)$ associated with the nonlinearity $g(x,y,|\cdot|)$, which we will define in Section \ref{sec:Musielak}, and $\delta^s$ is the two points finite difference $s$-quotient, with $0<s<1$, \[\delta^s u(x,y)=\frac{u(x)-u(y)}{|x-y|^s}.\]

We are mainly concerned with the corresponding obstacle problems
of the form
\begin{equation}\label{ObsPbGen}
u\in \mathbb{K}^s:\quad\langle\mathcal{L}_g^su-F,v-u\rangle\geq0\quad\forall v\in \mathbb{K}^s,
\end{equation} 
for $F\in W^{-s,G_{:}^*}(\Omega)$ and for the closed convex sets of one or two obstacles $\mathbb{K}^s=\mathbb{K}^s_1$, $\mathbb{K}^s_2$  defined, respectively, by 
\[\mathbb{K}_1^s=\{v\in W^{s,G_{:}}_0(\Omega):v\geq\psi \text{ a.e. in }\Omega\},\] 
\[\mathbb{K}_2^s=\{v\in W^{s,G_{:}}_0(\Omega):\psi\leq v\leq\varphi \text{ a.e. in }\Omega\},\]
with given functions $\psi,\varphi\in W^{s,G_{:}}(\mathbb{R}^d)$, supposing $\mathbb{K}^s_1\neq\emptyset$, for which it is sufficient to assume $\psi\leq0$ a.e. in $\mathbb{R}^d\backslash\Omega$, and $\mathbb{K}^s_2\neq\emptyset$, by assuming in addition that $\varphi\geq0$ a.e. in $\mathbb{R}^d\backslash\Omega$.

Here, $G_{:}(x,y,r)=\int_0^r g(x,y,\rho)\rho\,d\rho$ and $g(x,y,r):\mathbb{R}^d\times\mathbb{R}^d\times\mathbb{R}^+\to\mathbb{R}^+$ is a positive measurable function, Lipschitz continuous in $r$, such that, for almost every $x,y$, $\lim_{r\to0^+}rg(x,y,r)=0$, $\lim_{r\to +\infty}rg(x,y,r)=+\infty$  and 
\begin{equation}\label{gGrowthCond}0<g_*\leq \frac{r g'(x,y,r)}{g(x,y,r)}+1\leq g^*\quad\text{ for }r>0,\end{equation} 
for some constants $0<g_*\leq g^*$, as in \cite{AlbuquerqueAssisCarvalhoSalort2023-FracMusielakSobolevSpaces, AzroulBenkiraneShimiSrati2021-EmbeddingFractionalMusielakSobolev}.

Therefore $\mathcal{L}_g^s$ includes various nonlocal operators, as follows:
\begin{itemize}

    \item When $g(x,y,r)=g(r)$, we have the isotropic nonlinear nonlocal operator \begin{equation}\label{LgHomog}\int_{\mathbb{R}^d}\int_{\mathbb{R}^d}g\left(|\delta^s u(x,y)|\right)\delta^s u(x,y)\delta^s v(x,y)\frac{\,dx\,dy}{|x-y|^d},\end{equation}

    which corresponds to the fractional Orlicz-Sobolev case \cite{BonderSalort2019JFAFracOrliczSobolev} and, when $g=1$ is constant, includes the fractional Laplacian 
    \begin{equation}\label{La}\langle(-\Delta)^su, v\rangle = \int_{\mathbb{R}^d}\int_{\mathbb{R}^d}\frac{(u(x)-u(y))(v(x)-v(y))}{|x-y|^{d+2s}}\,dx\,dy;\end{equation}

    \item The anisotropic fractional $p$-Laplacian $\mathcal{L}_p^s$, for $1<p_*<p(x,y)<p^*<\infty$ (see for instance \cite{brasco2015stability,BahrouniJMAA2018ComparisonPrincipleFracPxLap}), corresponding to $g(x,y,r)=K(x,y)|r|^{p(x,y)-2}$  and defined through \begin{equation}\label{LpLap} \langle\mathcal{L}_p^su, v\rangle= \int_{\mathbb{R}^d}\int_{\mathbb{R}^d}\frac{|u(x)-u(y)|^{p(x,y)-2}(u(x)-u(y))(v(x)-v(y))}{|x-y|^{d+sp(x,y)}}K(x,y) \,dx\,dy,\end{equation} 
    where $K(x,y):\mathbb{R}^d\times\mathbb{R}^d\to\mathbb{R}$ is a measurable function satisfying 
\begin{equation}\label{KAssump}K(x,y)=K(y,x)\quad\text{ and }\quad k_*\leq K(x,y)\leq k^*\quad\text{ for a.e. }x,y\in\mathbb{R}^d\end{equation} for some $k_*,k^*>0$. In the linear case where $p=2$, we have the symmetric linear anisotropic fractional Laplacian (see, for instance, \cite{FracObsRiesz,RosOton2016Survey}); 

      \item The fractional double phase operator $\mathcal{L}_{p,q}^s$ corresponding to $g(x,y,r)=K_1(x,y)|r|^{p-2}+K_2(x,y)|r|^{q-2}$, or the logarithmic Zygmund operator with $g(x,y,r)=K_1(x,y)|r|^{p-2}+K_2(x,y)|r|^{p-2}|\log(|r|)|$, with $K_1,K_2$ satisfying \eqref{KAssump} (see, for instance, Example 2.3.2 of \cite{MusielakBook} for other $N$-functions); 
      
    \item We may also consider the special case of anisotropic operators of the type \eqref{Lgdef} with a strictly positive and bounded function $g(x,y,r)$ satisfying, in addition to \eqref{gGrowthCond},  \begin{equation}\label{tildegcond} 0<\gamma_*\leq g(x,y,r)\leq \gamma^*,\end{equation} for a.e. $x,y$ and for all $r$, which corresponds to the Hilbertian framework $H^s_0(\Omega)$ as in Chapter 5 of \cite{lo2022nonlocal}.

\end{itemize}

In recent years, there has been relevant progress in the study of PDEs in generalised Orlicz spaces including the obstacle problem (see, \cite{HarjulehtoHastoKlen2016GenOrliczPDE, MusielakBook, HarjulehtoKarppinen2024StabilityObstacleOrlicz} and their references), and also nonlocal operators in fractional generalised Orlicz spaces, also called fractional Musielak-Sobolev spaces, \cite{AzroulBenkiraneShimiSrati2022-NonlocalFractionalMusielakSobolev,AzroulBenkiraneShimiSrati2021-EmbeddingFractionalMusielakSobolev,AlbuquerqueAssisCarvalhoSalort2023-FracMusielakSobolevSpaces,ouali2024density}. The associated nonlocal elliptic equations in fractional generalised Orlicz spaces or the less general Orlicz-Sobolev spaces have also been extensively studied \cite{ByunKimOk2022MathAnnLocalHolderFracGLapAnisotropic,CarvalhoSilvaCarlosBahrouni2022BoundednessFracOrliczMoser,ChakerKimWeidner2022CVPDERegGLap,ChakerKimWeidner2022MathAnnHarnackFracAnisotropicGLap,BonderSalortVivas2021HolderEigenfunctionsGLap,BonderSalortVivas2020InteriorGlobalRegFracGLap,BonderSalort2019JFAFracOrliczSobolev,MolinaSalortVivas2021NonAFracGLapProperties}, including existence and regularity results, embedding and extension properties, local H\"older continuity, Harnack inequalities, and uniform boundedness properties. The associated unilateral problems have also been considered. Previous works along this line have only considered the fractional anisotropic $p$-Laplacian $\mathcal{L}_p^s$ in obstacle problems \cite{Palatucci2018NonAFracPLapObstacleHolder,K+Kuusi+Palatucci2016CVPDE-ObstaclePbFracPLap,Piccinini2022NonA-ObstaclePbFracPLap,Ok2023CV_HolderNonLocalVariable}. In this work, we consider the more general case of the anisotropic nonlocal nonlinear $g$-Laplacian $\mathcal{L}_g^s$ in generalised fractional Orlicz spaces, and we obtain new results for the associated obstacle problems.

This paper, extending the results of \cite{FracObsRiesz} to anisotropic nonlocal nonlinear operators, has the following plan:

\tableofcontents

In Section 2, after introducing the fractional generalised Orlicz functional framework for the operator $\mathcal{L}_g^s$, we recall some basic properties from the literature, as a Poincaré type inequality and some embedding results, in particular, in some fractional Sobolev-Gagliardo spaces. Then we state the existence of a unique variational solution to the homogeneous Dirichlet problem, which is a natural consequence of the assumptions on $g$ and the symmetry of the operator $\mathcal{L}_g^s$ and we prove a new global $L^\infty(\Omega)$ estimate, by using the truncation method used in \cite{PeralBoundedSolnEstimates} for the anisotropic fractional Laplacian. This global $L^\infty(\Omega)$ bound was obtained previously in the isotropic case of $g(x,y,r)=g(r)$ with $G$ satisfying the $\Delta'$ condition (which is stronger than the $\Delta_2$ condition) in Corollary 1.7 of \cite{CarvalhoSilvaCarlosBahrouni2022BoundednessFracOrliczMoser}, as well as Theorem 3 of \cite{BonderSalortVivas2021HolderEigenfunctionsGLap}, where these authors considered a different class of $G$, namely $G$ is such that $\bar{g}$ is convex and $g_*\geq1$ in \eqref{gGrowthCond}. We also collect some known regularity results with the aim to extend them to the solutions of the one and the two obstacles problems.

In Section 3, we first show that the structural assumption \eqref{gGrowthCond} implies that the $\mathcal{L}_g^s$ is a strictly T-monotone operator in $W^{s,G_{:}}_0(\Omega)$. This fact easily implies the monotonicity of the solution of the Dirichlet problem with respect to the data, extending and unifying previous results already known in some particular cases of $g$. This important property has interesting consequences in unilateral problems of obstacle type also in this generalised fractional framework: comparison of solution with respect to the data and a continuous dependence of the solutions in $L^\infty$ with respect to the $L^\infty$ variation of the obstacles; and more important, it also implies the Lewy-Stampacchia inequalities to this more general nonlocal framework, extending \cite{ServadeiValdinoci2013RMILewyStampacchia} and \cite{GMosconi} in the one obstacle case and are new in the nonlocal two obstacles problem. 

In the case when the heterogeneous term $f$ is in a suitable generalised Orlicz space, in Section 4, we give a direct proof of the Lewy-Stampacchia inequalities showing then that $\mathcal{L}_g^s u$ is also in the same Orlicz space. We also prove important consequences to the regularity of the solutions; and, in the case of integrable data, the approximation of the solutions via bounded penalisation. 

Finally, in Section 5, exploring the natural relation of the obstacle problem and potential theory, we make some elementary remarks on the extension of capacity to the fractional generalised Orlicz framework associated with the operator $\mathcal{L}_g^s$, motivating interesting open questions that are beyond the scope of this work. We refer to the recent work \cite{OrliczCapacities}, and its references, for the extension of the Sobolev capacity to generalised Orlicz spaces in the local framework of the gradient. We conclude this paper in the Hilbertian case of the anisotropic nonlinear operator \eqref{La}, with a few extensions relating the obstacle problem and potential theory, in the line of the pioneering work of Stampacchia \cite{StampacchiaEllipticEqsLmReg} for bilinear coercive forms, which was followed, for instance, in \cite{AdamsCapacityObstacle} and, in the nonlinear classical framework in \cite{AttouchPic_NPT_AFSToul_1979} and extended to the linear nonlocal setting in \cite{FracObsRiesz}.

Although we have considered only the nonlocal nonlinear anisotropic operators of the $g$-Laplacian type defined in the whole $\mathbb{R}^d$ by \eqref{Lgdef}, most of our results still hold in the different case in which the definition of the $g$-Laplacian type operator where the integral is instead taken only over the domain $\Omega$ as in \cite{KaufmannRossiVidal2017FracPxLap} and \cite{DelpezzoRossi2017TracesFracSobolevSpacesVariableExponents}.

\section{Preliminaries}

In this section we collect some known but dispersed facts, which can be found in the books \cite{HarjulehtoHastoLectures,LangMendezMonograph,KrasnoselskiiRutickii1961ConvexFunctionsOrliczSpaces,MusielakBook}, needed to develop our main results. After setting the functional framework of the fractional generalised Orlicz spaces we compile some relevant results on the fractional nonlinear Dirichlet problem in different cases.

\subsection{The Fractional Generalised Orlicz Functional Framework}\label{sec:Musielak}

Let the mapping $\bar{g}:\mathbb{R}^d\times\mathbb{R}^d\times\mathbb{R}^+\to\mathbb{R}$ be defined by 
\[\bar{g}(x,y,r)=g(x,y,r)r.\]
Then, with $g$ defined in the introduction, $\bar{g}$ satisfies the following condition:
\begin{enumerate}[label=(\arabic*)]
    \item $\bar{g}(x,y,\cdot):\mathbb{R}^d\times\mathbb{R}^d\times\mathbb{R}^+\to\mathbb{R}$ is a strictly increasing homeomorphism from $\mathbb{R}^+$ onto $\mathbb{R}$, $\bar{g}(x,y,r)>0$ when $r>0$.
\end{enumerate}
Moreover, its primitive $G_{:}=G(x,y,r):\mathbb{R}^d\times\mathbb{R}^d\times\mathbb{R}^+\to\mathbb{R}^+$ defined for all $r\geq0$ by \[G(x,y,r)=\int_0^r\bar{g}(x,y,\rho)\,d\rho\] satisfies:
\begin{enumerate}[resume,label=(\arabic*)]
    \item $G(x,y,\cdot):[0,\infty[\to\mathbb{R}$ is an increasing function, $G(x,y,0)=0$ and $G(x,y,r)>0$ whenever $r>0$;
    \item For the same constants $g_*<g^*$ as in \eqref{gGrowthCond}, \begin{equation}\label{GGrowthCond}0<1+g_*\leq \frac{r\bar{g}(x,y,r)}{G(x,y,r)}\leq g^*+1,\quad\text{ a.e. }x,y\in\mathbb{R}^d,\quad r\geq0,\end{equation} so $G_{:}$ satisfies the $\Delta_2$-condition, i.e. $G_{:}(2t)\leq CG_{:}(t)$ for $t>0$ and a.e. $x,y$, with a fixed $C>0$. (see \cite{Adams} or \cite{KrasnoselskiiRutickii1961ConvexFunctionsOrliczSpaces}.)
\end{enumerate}
The assumption \eqref{gGrowthCond} means that $G_{:}$ is a strictly convex function for a.e. $x,y$, and we denote $G^*_{:}=G^*(x,y,r):\mathbb{R}^d\times\mathbb{R}^d\times\mathbb{R}^+\to\mathbb{R}^+$ as the conjugate convex function of $G_{:}$, which is defined by  \[G^*(x,y,r)=\sup_{\rho>0}\{r\rho-G(x,y,\rho)\},\quad\forall x,y\in\mathbb{R}^d,r\geq0.\] 

In the example  $G(x,y,r)=\frac{1}{p(x,y)}|r|^{p(x,y)}$ corresponding to the anisotropic fractional $p$-Laplacian \eqref{LpLap}, we have $G^*(x,y,r)=\frac{1}{p'(x,y)}|r|^{p'(x,y)}$ with $\frac{1}{p(x,y)}+\frac{1}{p'(x,y)}=1$, for each $x,y\in\mathbb{R}^d$.

Given the function $G_{:}$, we can subsequently define the modulars $\Gamma_{\hat{G}_\cdot}$ and $\Gamma_{s,G}$ for $0<s<1$ and $u$ extended by 0 outside $\Omega$, following \cite{BonderSalort2019JFAFracOrliczSobolev}, by 
\[\Gamma_{\hat{G}_\cdot}(u)=\int_{\mathbb{R}^d}{\hat{G}_\cdot}(|u(x)|)\,dx,\]
\[\Gamma_{s,G_{:}}(u)=\int_{\mathbb{R}^d}\int_{\mathbb{R}^d}G_{:}\left(|\delta_su|\right)\frac{dx\,dy}{|x-y|^d}\quad\quad\text{ with }0<s<1,\]
where we denote \[\hat{G}_\cdot(r)=G(x,x,r),\] which also satisfies the global $\Delta_2$-condition.

Suppose we define the corresponding generalised Orlicz spaces and generalised fractional Orlicz-Sobolev spaces 
\[L^{\hat{G}_\cdot}(\mathbb{R}^d)=\left\{u:\mathbb{R}^d\to\mathbb{R},\text{measurable}:\Gamma_{\hat{G}_\cdot}(u)<\infty\right\},\] 
\[W^{s,G_{:}}(\mathbb{R}^d)=\left\{u\in L^{\hat{G}_\cdot}(\mathbb{R}^d):\Gamma_{s,G_{:}}(u)<\infty\right\}\] 
with their corresponding Luxemburg norms (see, for instance, Chapter 8 of \cite{Adams} or Chapter 2 of \cite{Musielak1983BookOrlicz}), given by 
\[\norm{u}_G=\norm{u}_{L^{\hat{G}_\cdot}(\mathbb{R}^d)}=\inf\left\{\lambda>0:\Gamma_{\hat{G}_\cdot}\left(\frac{u}{\lambda}\right)\leq1\right\}\] and \[\norm{u}_{s,G}=\norm{u}_{W^{s,G_{:}}(\mathbb{R}^d)}=\norm{u}_G+[u]_{s,G},\] where 
\[[u]_{s,G}=\inf\left\{\lambda>0:\Gamma_{s,G_{:}}\left(\frac{u}{\lambda}\right)\leq1\right\}.\] 
$L^{\hat{G}_\cdot}(\mathbb{R}^d)$ and $W^{s,G_{:}}(\mathbb{R}^d)$ are reflexive Banach spaces by the $\Delta_2$-condition (refer to Theorem 11.6 of \cite{Musielak1983BookOrlicz}). 

As in Lemmas 3.1 and 3.3 of \cite{AzroulBenkiraneShimiSrati2022-NonlocalFractionalMusielakSobolev}), the strictly convex functional  $\Gamma_{s,G_{:}}\in C^1(W^{s,G_{:}}(\mathbb{R}^d),\mathbb{R})$ and is weakly lower semi-continuous.

We define 
\[W^{s,G_{:}}_0(\Omega)= \overline{C_c^\infty(\Omega)}^{\norm{\cdot}_{s,G}}\]
with dual $[W^{s,G_{:}}_0(\Omega)]^*=W^{-s,G_{:}^*}(\Omega)$, as $G_{:}$ satisfies the $\Delta_2$-condition (see Sections 3.3 and 3.5 of \cite{MusielakBook}), and we consider a function $v\in W^{s,G_{:}}_0(\Omega)$ defined everywhere in $\mathbb{R}^d$ by setting $v=0$ in $\mathbb{R}^d \backslash \Omega$. Furthermore, by Lemma 2.5.5 of \cite{LangMendezMonograph}, $C_c^\infty(\Omega)$ is dense in $C(\Omega)\cap L^{\hat{G}_\cdot}(\Omega)$. 

We denote by $\hat{G}^{-1}_\cdot(r)=G^{-1}(x,x,r)$ the inverse function of $\hat{G}_\cdot$ for almost all $x$, which satisfies the following conditions:
\begin{equation}
    \int_0^1\frac{\hat{G}^{-1}_\cdot(t)}{t^{(d+s)/d}}\,dt<\infty, \quad \int_1^\infty\frac{\hat{G}^{-1}_\cdot(t)}{t^{(d+s)/d}}\,dt=\infty \quad \text{ for almost all }x\in\Omega.
\end{equation}
Then, the inverse generalised Orlicz conjugate function of $\hat{G}_\cdot$ is defined as
\begin{equation}
    (\tilde{G}_\cdot)^{-1}(r)=\int^r_0\frac{\hat{G}^{-1}_\cdot(t)}{t^{(d+s)/d}}\,dt \quad \text{ for almost all }x\in\Omega.
\end{equation} 
Then, by Theorem 2.1 of \cite{AzroulBenkiraneShimiSrati2021-EmbeddingFractionalMusielakSobolev}, the embeddings $W^{s,G_{:}}_0(\Omega)\hookrightarrow L^{\tilde{G}_\cdot}(\Omega)$ and $[L^{\tilde{G}_\cdot}(\Omega)]^*\hookrightarrow W^{-s,G_{:}^*}(\Omega)$ hold for the bounded open subset $\Omega\subset\mathbb{R}^d$ with Lipschitz boundary. 
For any $F\in W^{-s,G_{:}^*}(\Omega)$ and $u\in W^{s,G_{:}}_0(\Omega)$, we denote their inner product by $\langle\cdot,\cdot\rangle$. As $\tilde{G}_\cdot$ also satisfies the 
$\Delta_2$-condition, we have $[L^{\tilde{G}_\cdot}(\Omega)]^*=L^{\hat{G}^*_\cdot}(\Omega)$ and so when $F=f\in L^{\hat{G}^*_\cdot}(\Omega)$, then 
\begin{equation}\label{IntegralForm}\langle f,u\rangle = \int_{\Omega} fu\,dx \quad\forall u\in L^{\tilde{G}_\cdot}(\Omega).\end{equation} 
Furthermore, we have a Poincaré type inequality:

\begin{lemma}[Corollary of Theorem 2.3 of \cite{AzroulBenkiraneShimiSrati2021-EmbeddingFractionalMusielakSobolev}]\label{Poincare}
    Let $s\in]0,1[$ and $\Omega$ be a bounded open subset of $\mathbb{R}^d$ with a Lipschitz bounded boundary. Then there exists a constant $C=C(s,d,\Omega)>0$ such that \[\norm{u}_{L^{\hat{G}_\cdot}(\Omega)}\leq C[u]_{s,G_{:}}\] for all $u\in W^{s,G_{:}}_0(\Omega)$. Therefore, the embedding \begin{equation}W^{s,G_{:}}_0(\Omega)\hookrightarrow L^{\hat{G}_\cdot}(\Omega)\end{equation} is continuous.
    Furthermore, $[u]_{s,G}$ is an equivalent norm to $\norm{u}_{s,G}$ for the fractional generalised Orlicz space $W^{s,G_{:}}_0(\Omega)$.
\end{lemma}
\begin{remark}
    Note that in the bounded open sets $\Omega$, the spaces we consider here are different from the $W^{s,G_{xy}}(\Omega)$ spaces considered in \cite{AzroulBenkiraneShimiSrati2022-NonlocalFractionalMusielakSobolev,AzroulBenkiraneShimiSrati2021-EmbeddingFractionalMusielakSobolev,AlbuquerqueAssisCarvalhoSalort2023-FracMusielakSobolevSpaces}, defined by
    \[W^{s,G_{xy}}(\Omega)=\left\{u\in L^{\hat{G}_x}(\Omega):\Phi_{s,G_{xy}}(u)<\infty\right\}\] where,  for $0<s<1$, 
    \begin{equation*}\Phi_{s,G_{xy}}(u)=\int_\Omega\int_\Omega G_{xy}\left(|\delta_su|\right)\frac{dx\,dy}{|x-y|^d}\end{equation*}
    with $G_{xy}:\Omega\times\Omega\times\mathbb{R}^+\to\mathbb{R}^+$ is defined only for a.e. $(x,y)\in \Omega\times\Omega$ with similar properties to our $G_{:}:\mathbb{R}^d\times\mathbb{R}^d\times\mathbb{R}^+\to\mathbb{R}^+$. We noticed that by Remark 2.2 of \cite{AzroulBenkiraneShimiSrati2022-NonlocalFractionalMusielakSobolev}] it is known $C_c^\infty(\Omega)\subset C_c^2(\Omega)\subset W^{s,G_{xy}}(\Omega)$. 
    
    Since the spaces we consider are, in a certain sense, smaller than the $W^{s,G_{xy}}(\Omega)$ spaces, as $W^{s,G_{:}}_0(\Omega)\hookrightarrow W^{s,G_{xy}}_0(\Omega)$ the embedding results in \cite{AzroulBenkiraneShimiSrati2022-NonlocalFractionalMusielakSobolev,AzroulBenkiraneShimiSrati2021-EmbeddingFractionalMusielakSobolev,AlbuquerqueAssisCarvalhoSalort2023-FracMusielakSobolevSpaces} still hold,  as Lemma \ref{Poincare} above.

    Observe that the space $L^{\tilde{G}_x}(\Omega)$ defined with 
    \[\Phi_{\hat{G}_x}(u)=\int_{\Omega}{\hat{G}_x}(|u(x)|)\,dx\] for $\hat{G}_x(x)=G_{xy}(x,x)$  is the same as $L^{\tilde{G}_\cdot}(\Omega)$.
\end{remark}

\begin{remark}
    In the case  $\Omega=\mathbb{R}^d$, $W^{s,G_{:}}(\mathbb{R}^d)$ and $W^{s,G_{xy}}(\mathbb{R}^d)$ coincide.
\end{remark}

For completeness, we also register the following properties.

\begin{lemma}\label{EmbeddingLemma}$ $
    \begin{itemize}
        \item \emph{[Theorem 3.3 of \cite{ouali2024density}].} $C_c^\infty(\mathbb{R}^d)$ is dense in $W^{s,G_{:}}(\mathbb{R}^d)$, so $W^{s,G_{:}}(\mathbb{R}^d)=W^{s,G_{:}}_0(\mathbb{R}^d)$.
        \item \emph{[Proposition 2.1 of \cite{AzroulBenkiraneShimiSrati2021-EmbeddingFractionalMusielakSobolev}].} For a bounded open subset $\Omega\subset\mathbb{R}^d$ and $0<s_1\leq s\leq s_2<1$, the embeddings 
        \[W^{s_2,G_{:}}_0(\Omega)\hookrightarrow W^{s,G_{:}}_0(\Omega)\hookrightarrow W^{s_1,G_{:}}_0(\Omega)\] are continuous. 
    \end{itemize}
\end{lemma}


Furthermore, for bounded domains $\Omega\subset\mathbb{R}^d$, 
\begin{equation}\label{LGLpApproxEmbedOmega}L^{g^*+1}(\Omega)\subset L^{\hat{G}_\cdot}(\Omega)\subset L^{g_*+1}(\Omega),\end{equation}
which is also a consequence of Theorem 8.12 (b) of \cite{Adams} and the inequality
\begin{multline*}\log(r^{1+g_*})-\log(r_0^{1+g_*})=\int_{r_0}^r\frac{1+g_*}{r}\,dr\leq\int_{r_0}^r \frac{\bar{g}(x,y,r)}{G(x,y,r)}\,dr\\=\log(G(x,y,r))-\log(G(x,y,r_0))\leq \log(r^{1+g^*})-\log(r_0^{1+g^*})\end{multline*} 
that holds for every $0<r_0<r$, by assumption \eqref{GGrowthCond}. In fact, this means $G(x,y,r)$ dominates $r^{g_* +1}$ and is dominated by $r^{g^* +1}$ as $r\to \infty$ and the embeddings \eqref{LGLpApproxEmbedOmega} follow.

We recall the definition of the fractional Sobolev-Gagliardo spaces $W^{s,p}_0(\Omega)$ as the closure of  $C_c^\infty(\Omega)$ in
\[W^{s,p}(\Omega)=\left\{u\in L^p(\Omega) :  [u]_{s,p,\Omega}^p=\int_{\Omega} \int_{\Omega} \frac{| u(x)- u(y)|^{p}}{|x-y|^{sp}}\frac{ dx\, dy}{|x-y|^d}< \infty  \right\}.\]

Then, we have
\begin{proposition}[Lemma 2.3 of \cite{AzroulBenkiraneShimiSrati2021-EmbeddingFractionalMusielakSobolev}]\label{WGWspApproxEmbedProp}
    For any $0<s<1$ and $\Omega\subset\mathbb{R}^d$ open bounded subset, 
    \begin{equation}\label{WGWspApproxEmbed} W^{s,G_{:}}_0(\Omega)\hookrightarrow  W^{t,q}_0(\Omega)\quad \text{ for any }0<t<s,1\leq q<1+g_*.\end{equation}
\end{proposition}

In addition, we can combine the embedding \eqref{WGWspApproxEmbed} and the classical Rellich-Kondrachov compactness embedding we have $W^{t,q}_0(\Omega)\subset L^{q^*}(\Omega)$ with $q^*$ satisfying $1\leq q^* < \frac{dq}{d-tq} < \frac{d(g_*+1)}{d-s(g_*+1)}$. Observe that it is necessary that $s(g_*+1)<d$. This embedding result is given as follows:

\begin{corollary}\label{CompactnessTheoremRellichLp}
$W^{s,G_{:}}_0(\Omega)\Subset L^q(\Omega)$ with $q$ satisfying $1\leq q < \frac{d(g_*+1)}{d-s(g_*+1)}$. 
\end{corollary}

\begin{remark}
Observe that in the Hilbertian framework of \eqref{tildegcond} the Banach space $W^{s,G_{:}}_0(\Omega)$, with the assumption \eqref{tildegcond} is algebraically and topologically) equivalent to the fractional  Sobolev space $H^s_0(\Omega)=W^{s,2}_0(\Omega)$, which is a Hilbert space, while $W^{s,G_{:}}_0(\Omega)$ is not. 
\end{remark}

\subsection{The Quasilinear Fractional Dirichlet Problem}\label{sec:DirichletPb}

Recalling that $G_{:}$ is a strictly convex and differentiable function in $r$ for a.e. $x,y$, we can regard $\mathcal{L}_g^s$ as the potential operator with respect to the convex functional \begin{equation}\label{LgFuncPotential}\Gamma_{s,G_{:}}(v)=\int_{\mathbb{R}^d}\int_{\mathbb{R}^d}G_{:}\left(|\delta^s v|\right)\frac{dx\,dy}{|x-y|^d}.\end{equation}  

As a consequence of well known results of convex analysis, there exists a unique solution to the Dirichlet problem, given formally by $\mathcal{L}_g^su=F$ in $\Omega$, $u=0$ in $\Omega^c$.

\begin{proposition}\label{EllipticDirichletExistThm}[Proposition 4.6 of \cite{AlbuquerqueAssisCarvalhoSalort2023-FracMusielakSobolevSpaces}]
Let $0<s<1$ and $\Omega\subset\mathbb{R}^d$ be a bounded domain. For $F\in W^{-s,G_{:}^*}(\Omega)$, there exists a unique variational solution $u\in W^{s,G_{:}}_0(\Omega)$ to 
\begin{equation}\label{EllipticProbEq}\langle \mathcal{L}_g^su,v\rangle = \langle F,v\rangle \quad\forall v\in W^{s,G_{:}}_0(\Omega),\end{equation} which is equivalent to the minimum over $W^{s,G_{:}}_0(\Omega)$ of the functional $\mathcal{G}_s:W^{s,G_{:}}_0(\Omega)\to\mathbb{R}$ defined by \begin{equation}\label{LgFunc}\mathcal{G}_s(v)=\int_{\mathbb{R}^d}\int_{\mathbb{R}^d}G_{:}\left(|\delta^s v|\right)\frac{dx\,dy}{|x-y|^d}-\langle F,v\rangle\quad\forall v\in W^{s,G_{:}}_0(\Omega).\end{equation}
\end{proposition}

In the next Theorem we extend the global boundedness of the solutions for the anisotropic Dirichlet problem, under the uniform assumption \eqref{gGrowthCond} on $g$. 


\begin{theorem}\label{UnifBddThmMusielak}
Suppose $F=f\in L^m(\Omega)$, with $m > \frac{d}{s(g_*+1)}$ and $g$ satisfies \eqref{gGrowthCond} with 
$s(g_*+1)<d$. Let $u$ denote the solution of the Dirichlet problem \eqref{EllipticProbEq}. Then there exists a constant $C$, depending only on $g_*$, $g^*$, $k_*$, $k^*$, $d$, $\Omega$, $\norm{u}_{W^{s,G_{:}}_0(\Omega)}$, $\norm{f}_{L^m(\Omega)}$ and $s$, such that \[\norm{u}_{L^\infty(\Omega)}\leq C.\] 
\end{theorem}

The proof extends the one given in Section 3.1.2 of \cite{PeralBoundedSolnEstimates}. It uses the following numerical iteration estimate, the proof of which is given in Lemma 4.1 of \cite{StampacchiaEllipticEqsLmReg}.

\begin{lemma}\label{StampacchiaLemma}
    Let $\Psi:\mathbb{R}^+\to\mathbb{R}^+$ be a nonincreasing function such that 
    \[\Psi(h)\leq \frac{M}{(h-k)^\gamma}\Psi(k)^\delta\quad\forall h>k>0,\] where $M,\gamma>0$ and $\delta>1$. Then $\Psi(d)=0$, where $d^\gamma = M\Psi(0)^{\delta-1}2^\frac{\delta\gamma}{\delta-1}$.
\end{lemma}

Next, we introduce the truncation function $T_k$ and its complement $P_k$ defined as 
\[T_k(u)=-k\vee(k\wedge u),\quad P_k(u)= u-T_k(u)\quad \text{ for every }k\geq0,\] which will be useful for the proof. 

Given the above definitions of $T_k$ and $P_k$, it is straightforward to see (by considering the cases of $v(x),v(y)$ $\geq k$ and $\leq k$) that \begin{equation}\label{TruncationEst}[T_k(v(x))-T_k(v(y))][P_k(v(x))-P_k(v(y))]\geq0\quad\text{ a.e. in }\Omega\times\Omega.\end{equation} 

As a result, we have under the assumptions of this theorem, the following Lemma.
\begin{lemma}
    Take $v\in W^{s,G_{:}}_0(\Omega)$. If $\Psi:\mathbb{R}\to\mathbb{R}$ is a Lipschitz function such that $\Psi(0)=0$, then $\Psi(v)\in W^{s,G_{:}}_0(\Omega)$. In particular, for any $k\geq0$, $T_k(v),P_k(v)\in W^{s,G_{:}}_0(\Omega)$, and 
    \[(g_*+1)\Gamma_{s,G_{:}}(P_k(v))\leq \langle \mathcal{L}^s_gv, P_k(v)\rangle.\]
\end{lemma}

\begin{proof}
    We first show the regularity of $T_k(v)$ and $P_k(v)$. Let $\lambda_\Psi>0$ be the Lipschitz constant of $\Psi$. As such, for $x, y$ in     $\mathbb{R}^d, x\neq y$,
    \[|\delta^s \Psi(v)(x,y)| = \frac{|\Psi(v(x)) - \Psi(v(y))|}{|x-y|^s} \leq \lambda_\Psi \frac{|v(x) - v(y)|}{|x-y|^s} = \lambda_\Psi |\delta^s v(x,y)|.\] 
    Since $r\mapsto rg(\cdot,\cdot,r)$ is monotone increasing, as a result of the assumption \eqref{gGrowthCond}, we have that
    \[|\delta^s \Psi(v)|g\left(x,y,|\delta^s \Psi(v)|\right)\leq |\lambda_\Psi \delta^s v|g\left(x,y,|\lambda_\Psi \delta^s v|\right)\] 
    for a.e. $x, y$ in     $\mathbb{R}^d$, and so
    \begin{align}
    (g_*+1)\Gamma_{s,G_{:}}(\Psi(v)) & \leq \int_{\mathbb{R}^d}\int_{\mathbb{R}^d}g\left(x,y,|\delta^s \Psi(v)|\right) |\delta^s \Psi(v)|^2 \frac{dx\,dy}{|x-y|^d}\label{UnifBddEst1}\\
    & \leq \int_{\mathbb{R}^d}\int_{\mathbb{R}^d}g\left(x,y,|\lambda_\Psi \delta^s v|\right) |\lambda_\Psi \delta^s v|^2 \frac{dx\,dy}{|x-y|^d} \leq (g^*+1)\lambda_\Psi^2 \Gamma_{s,G_{:}}(\lambda_\Psi v)\nonumber
    \end{align} by \eqref{GGrowthCond}. Then, the regularity of $T_k(v)$ and $P_k(v)$ follows since $T_k$ and $P_k$ are Lipschitz functions with Lipschitz constant 1.

    Finally we consider $\langle \mathcal{L}^s_gv, P_k(v)\rangle$. Since $P_k$ is a monotone Lipschitz function with Lipschitz constant 1, we can apply a similar argument as above to obtain that 
    \begin{align*}
    \langle \mathcal{L}^s_gv, P_k(v)\rangle & = \int_{\mathbb{R}^d}\int_{\mathbb{R}^d}g\left(x,y,|\delta^s v|\right) \delta^s v \, \delta^s P_k(v) \frac{dx\,dy}{|x-y|^d}\\
    & \geq \int_{\mathbb{R}^d}\int_{\mathbb{R}^d}g\left(x,y,|\delta^s P_k(v)|\right) \delta^s P_k(v) \delta^s v \frac{dx\,dy}{|x-y|^d} = \langle \mathcal{L}^s_g P_k(v), v \rangle
    \end{align*}
    since $g$ is non-negative and 
    \begin{align*}
        \delta^s v \, \delta^s P_k(v) & = \frac{P_k(v(x)) - P_k(v(y))}{|x-y|^s} \frac{v(x) - v(y)}{|x-y|^s}  \\
        & = \frac{\left(P_k(v(x)) - P_k(v(y))\right)^2 + \left(T_k(v(x)) - T_k(v(y))\right)\left(P_k(v(x)) - P_k(v(y))\right)}{|x-y|^{2s}}\\
        & \geq \frac{\left(P_k(v(x)) - P_k(v(y))\right)^2 }{|x-y|^{2s}} >0,
    \end{align*} by recalling that $v=T_k(v)+P_k(v)$ as well as using the estimate \eqref{TruncationEst}.
    Using this inequality, we therefore have 
    \begin{align*}
        \langle \mathcal{L}^s_gv, P_k(v)\rangle & \geq \langle \mathcal{L}^s_g P_k(v), v \rangle = \int_{\mathbb{R}^d}\int_{\mathbb{R}^d}g\left(x,y,|\delta^s P_k(v)|\right) \delta^s P_k(v) \delta^s v \frac{dx\,dy}{|x-y|^d} \\
        & \geq \int_{\mathbb{R}^d}\int_{\mathbb{R}^d}g\left(x,y,|\delta^s P_k(v)|\right) \left(P_k(v(x)) - P_k(v(y))\right)^2 \frac{dx\,dy}{|x-y|^{d+2s}},
    \end{align*}
    hence the desired result by \eqref{UnifBddEst1}.
\end{proof}

Making use of the above estimates, we prove the uniform boundedness of the unique solution to the nonlinear Dirichlet problem.

\begin{proof}[Proof of Theorem \ref{UnifBddThmMusielak}]
We take $P_k(u)$ to be the test function in the variational formulation of \eqref{EllipticProbEq}. 
Combining this with the previous lemma, we easily obtain that
\[(g_*+1)\Gamma_{s,G_{:}}(P_k(u(x))) \leq \langle \mathcal{L}^s_gu(x), P_k(u(x))\rangle = \int_{A_k}f(x)P_k(u(x))\,dx,\] where $A_k=\{x\in\Omega:u\geq k\}$. 

To estimate the left-hand-side, we make use of the inclusion of $W^{s,G_{:}}(\Omega)\hookrightarrow W^{t,q}(\Omega)$ spaces. Then 
\[\Gamma_{s,G_{:}}(P_k(u(x)))\geq C\norm{P_k(u(x))}_{W^{t,q}_0(\Omega)}^q \geq C'\norm{P_k(u(x))}_{L^{q^*}(\Omega)}^q\] for an embedding constant $C$ and exponent $q=1+g_*-\epsilon$ of \eqref{WGWspApproxEmbed} for some small $\epsilon>0$, and Sobolev embedding constants $C'/C$ and $t,q^*$ of Corollary \ref{CompactnessTheoremRellichLp} (see, for instance, Theorem 6.5 of \cite{HitchhikerGuide}). 

To estimate the right-hand-side, we apply the H\"older's inequality. Then, for any $m>0$, we have
\[\left|\int_{A_k}f(x)P_k(u(x))\,dx\right| \leq \norm{f}_{L^m(\Omega)}\norm{P_k(u(x))}_{L^{q^*}(\Omega)}|A_k|^{1-\frac{1}{q^*}-\frac{1}{m}}.\]

Combining these estimates with the crucial observation that for any $h>k$, $A_h\subset A_k$ and $P_k(u)\chi_{A_h}\geq h-k$, we obtain that
\[(h-k)|A_h|^{\frac{g_*-\epsilon}{q^*}}\leq \frac{1}{k_* C' (g_*+1-\epsilon)} \norm{f}_{L^m(\Omega)}|A_k|^{1-\frac{1}{q^*}-\frac{1}{m}},\]
or
\[|A_h|\leq \frac{C''}{(h-k)^{\frac{q^*}{g_*-\epsilon}}} \norm{f}_{L^m(\Omega)}^{\frac{q^*}{g_*}}|A_k|^{\frac{q^*}{g_*-\epsilon}(1-\frac{1}{q^*}-\frac{1}{m})}\] for a constant $C''>0$.

Finally, observe that for $m > \frac{d}{s(g_*+1)}$,
\[\frac{q^*}{g_*-\epsilon}\left(1-\frac{1}{q^*}-\frac{1}{m}\right) > 1\] for large enough $q^*$ and small enough $\epsilon>0$. Therefore, the assumptions of Lemma \eqref{StampacchiaLemma} above are all satisfied, and we can take $\Psi(h)=|A_h|$ in Lemma \eqref{StampacchiaLemma} to obtain that there exists a $k_0$ such that $\Psi(k)\equiv0$ for all $k\geq k_0$, thus $\esssup_\Omega u \leq k_0$.
\end{proof}



\begin{remark}
    Note that the assumption \eqref{gGrowthCond} implies that $G_{:}$ satisfies the $\Delta_2$ condition, which is weaker than the $\Delta'$ condition given by 
    \begin{equation}\label{DeltaPrime}G_{:}(rt)\leq CG_{:}(r)G_{:}(t), \quad \text{ for }r,t>0\text{ and some }C>0,\end{equation} 
    and used in the $L^\infty$-estimate in \cite{CarvalhoSilvaCarlosBahrouni2022BoundednessFracOrliczMoser}.
\end{remark}

Recently in the case of the fractional $p(x,y)$-Laplacian an interesting local H\"older regularity result for the solution of the Dirichlet problem has been proved, extending previous results in the case of constant $p$. Here $C^\alpha(\omega)$ denotes the space of H\"older continuous functions in $\omega$ for some $0<\alpha<1$.

\begin{theorem}\label{HolderRegMusielak}
Let $F=f\in L^\infty(\Omega)$. Suppose $g(x,y,r)$ is of the form $|r|^{p(x,y)-2}K(x,y)$ as in the fractional $p$-Laplacian $\mathcal{L}^s_p$ in \eqref{LpLap} for $1<p_-\leq p(x,y)\leq p_+<\infty$, and $K$ satisfies \eqref{KAssump}, with $p(\cdot,\cdot)$ and $K(\cdot,\cdot)$ symmetric. 
\begin{enumerate}[label=(\alph*)]
    \item Suppose further that $p(x,y)$ is log-H\"older continuous on the diagonal $D=\{(x,x):x\in\Omega\}$, i.e. 
\[\sup_{0<r\leq 1/2}\left[\log\left(\frac{1}{r}\right)\sup_{B_r\subset\Omega}\sup_{x_2,y_1,y_2\in B_r}|p(x_1,y_1)-p(x_2,y_2)|\right]\leq C \quad\text{ for some }C>0.\] Then, the solution $u$ of the Dirichlet problem \eqref{EllipticProbEq} is locally H\"older continuous, i.e. \[u\in C^\alpha(\Omega)\text{ for some }0<\alpha<1.\] 

\item In the case where $p_-=p_+=p$
, the solution $u$ of \eqref{EllipticProbEq} is globally H\"older continuous and satisfies \begin{equation}u\in C^\alpha(\bar{\Omega})\quad\text{ such that }\quad  \norm{u}_{C^\alpha(\bar{\Omega})}\leq C_s\end{equation} for some $0<\alpha<1$ depending on $d$, $p$, $s$, $g_*$, $g^*$, $k_*$, $k^*$ and $\norm{f}_{L^\infty(\Omega)}$.

\end{enumerate}
\end{theorem}

\begin{remark} 
Part (a) of this result is given in Theorem 1.2 of \cite{Ok2023CV_HolderNonLocalVariable}. 

Part (b), when $p$ is constant and the anisotropy is in the kernel $K$, is the result given in Theorem 8 of \cite{Palatucci2018NonAFracPLapObstacleHolder} or Theorem 6 of \cite{K+Kuusi+Palatucci2016CVPDE-ObstaclePbFracPLap}, and extended in Theorem 1.3 of \cite{Piccinini2022NonA-ObstaclePbFracPLap} to the Heisenberg group. 
\end{remark}

Recalling that $L^\infty(\Omega)\subset L^{\hat{G}^*_\cdot}(\Omega)$ by \eqref{LGLpApproxEmbedOmega}, next we compile the following known regularity results for the Dirichlet problem for the operator $\mathcal{L}^s_g$ under the more restrictive assumption on $G$ being isotropic, i.e. in the Orlicz-Sobolev case. 
\begin{theorem}\label{EllipticDirichletRegThm1}
Let $u$ be the solution of the Dirichlet problem \eqref{EllipticProbEq}. Suppose $g$ is isotropic, i.e. $g=g(r)$ is independent of $(x,y)$ and $F=f\in L^\infty(\Omega)$.
\begin{enumerate}[label=(\alph*)]

    \item  If $G$ satisfies the $\Delta'$ condition, 
   then the solution $u$ of \eqref{EllipticProbEq} is such that $u\in C^{\alpha}_{loc}(\Omega)$ for some $0<\alpha<1$ depending on $d$, $s$, $g_*$ and $g^*$, and there exists $C_\omega>0$ for every $\omega\Subset\Omega$ depending only on $d$, $g_*$ and $g^*$, $\norm{f}_{L^\infty(\Omega)}$ and independent of $s\geq s_0>0$, such that, for some for $0<\alpha\leq s_0$, 
    \begin{equation}u\in C^\alpha(\omega)\quad\text{ with }\quad  \norm{u}_{C^\alpha(\omega)}\leq C_\omega.\end{equation} 

    \item If $\bar{g}=\bar{g}(r)$ is convex in $r$ and $g_*\geq1$, then $u$ is H\"older continuous up to the boundary, i.e.  
    \begin{equation}\label{uRegBonderSalortVivas}u\in C^\alpha(\bar{\Omega})\quad\text{ such that }\quad  \norm{u}_{C^\alpha(\bar{\Omega})}\leq C_s\end{equation} for $\alpha\leq s$
    where $C_s>0$ and $\alpha>0$ depends only on $s$, $d$, $g_*$, $g^*$ and $\norm{f}_{L^\infty(\Omega)}$.
\end{enumerate}
\end{theorem}
    
\begin{remark}\label{Remark1}
    Part (a) of this result is obtained in Theorem 1.1 of \cite{ByunKimOk2022MathAnnLocalHolderFracGLapAnisotropic} and in Theorem 1.1(i) of \cite{ChakerKimWeidner2022CVPDERegGLap}. Note that in these references, the authors require that the tail function of $u$ for the ball $B_R(x_0)$ defined by 
    \[Tail(u;x_0,R)=\int_{\mathbb{R}^d\backslash B_R(x_0)} \bar{g}\left(\frac{|u(x)|}{|x-x_0|^s}\right)\frac{dx}{|x-x_0|^{n+s}}\]
    is bounded. This assumption is not necessary when we apply it to the Dirichlet problem \eqref{EllipticProbEq}, since the solution $u$ is globally bounded by Theorem \ref{UnifBddThmMusielak}, and therefore its tail is also bounded.

Part (b) of this result is Theorem 1.1 of \cite{BonderSalortVivas2020InteriorGlobalRegFracGLap}.  The additional assumption $g_*\geq1$ implies that, in the case of the fractional $p$-Laplacian $\mathcal{L}_{p}^s$
the result only covers the degenerate constant case $p\geq2$. 
\end{remark}

\begin{theorem}\label{EllipticDirichletRegThm2}
    Let $u$ be the solution of the Dirichlet problem \eqref{EllipticProbEq}. Suppose $g(x,y,r)$ is uniformly bounded and positive as in \eqref{tildegcond}.
\begin{enumerate}[label=(\alph*)]
    \item  Let $f\in L^q_{loc}(\Omega)$ for some $q>\frac{2d}{d+2}$. Then, there exists a positive $0<\delta<1-s$ depending on $d$, $s$, $g_*$, $g^*$, $q$ independent of the solution $u$, such that $u\in W^{s+\delta,2+\delta}_{loc}(\Omega)$. 
    
    \item Suppose further that $f\in L^\infty(\Omega)$ and $g(x,y,r)=g(y,x,r)$, i.e. $g$ has symmetric anisotropy, the solution $u$ of \eqref{EllipticProbEq} is also globally H\"older continuous and satisfies \eqref{uRegBonderSalortVivas} for some $0<\alpha<1$ depending on $d$, $p$, $s$, $\gamma_*$ and $\gamma^*$.
\end{enumerate}
\end{theorem}

\begin{remark}
Part (a) of this result is obtained by applying the result of Theorem 1.1 of \cite{KuusiMingioneSire2015APDE-RegularityNonlocal} by replacing the kernel $K(x,y)$ with the bounded kernel $g(x,y,|\delta_s u(x,y)|)$ satisfying \eqref{tildegcond}, being $u$ the solution of the nonlinear Dirichlet problem \eqref{EllipticProbEq}. 

Part (b) of this result in the special case when $g(x,y,r)$ is uniformly bounded, in the sense that $0<\gamma_*\leq g(x,y,r)\leq \gamma^*$, is a simple corollary of Theorem \ref{HolderRegMusielak} in the case $p=2$, since $|\delta_su|$ is symmetric and we can consider $g(x,y,|\delta_s u(x,y)|)=K(x,y)$ as a function of $x$ and $y$ for the regularity estimate.
\end{remark}

\section{Quasilinear Fractional Obstacle Problems}

Exploring the order properties of the fractional generalised Orlicz spaces and showing the T-monotonicity property in this large class of nonlocal operators, we are able to extend well-known properties to the fractional framework: comparison of solution with respect to the data and the Lewy-Stampacchia inequalities for obstacle problems.

\subsection{T-monotonicity and Comparison Properties}

We start by showing that the quasilinear fractional operator $\mathcal{L}_g^s$ is strictly T-monotone in $W^{s,G_{:}}_0(\Omega)$, i.e. \[\langle\mathcal{L}_g^su-\mathcal{L}_g^sv,(u-v)^+\rangle>0\quad\forall u\neq v.\] Here, we use the standard notation for the positive and negative parts of $v$ 
\[v^+\equiv v\vee0\quad\text{ and }\quad v^-\equiv-v\vee0=-(v\wedge0),\] and we recall the Jordan decomposition of $v$ given by \[v=v^+-v^-\quad\text{ and }\quad |v|\equiv v\vee(-v)=v^++v^-\] and the useful identities \[u\vee v=u+(v-u)^+=v+(u-v)^+,\]\[u\wedge v=u-(u-v)^+=v-(v-u)^+.\]

\begin{theorem}\label{LgTMonotone}
The operator $\mathcal{L}_g^s$ is strictly T-monotone in $W^{s,G_{:}}_0(\Omega)$. 
\end{theorem}

\begin{proof}

Setting $\theta_r(x,y)=r\delta^s u(x,y)+(1-r)\delta^s v(x,y)$ and writing $w=u-v$, we have 
\begin{align*}&\quad\langle\mathcal{L}_g^su-\mathcal{L}_g^sv,w^+\rangle\\&=\int_{\mathbb{R}^d}\int_{\mathbb{R}^d}(w^+(x)-w^+(y))\left[g_{:}\left(|\delta^s u|\right)\delta^s u-g_{:}\left(|\delta^s v|\right)\delta^s v\right]\frac{dy\,dx}{|x-y|^{d+s}}\\&=\int_{\mathbb{R}^d}\int_{\mathbb{R}^d}(w^+(x)-w^+(y))\left[\int_0^1g_{:}(|\theta_r|)\,dr+\int_0^1|\theta_r|g'_{:}(|\theta_r|)\,dr\right](\delta^s u-\delta^s v)\frac{dy\,dx}{|x-y|^{d+s}}\end{align*} 
Now, by \eqref{gGrowthCond}, 
\[J(x,y)=\left[\int_0^1g_{:}(|\theta_r|)\,dr+\int_0^1|\theta_r|g'_{:}(|\theta_r|)\,dr\right]>0\] is strictly positive and bounded, so we have
\begin{align*}\langle\mathcal{L}_g^su-\mathcal{L}_g^sv,(u-v)^+\rangle&=\int_{\mathbb{R}^d}\int_{\mathbb{R}^d}J(x,y)\frac{w^+(x)-w^-(x)-w^+(y)+w^-(y)}{|x-y|^{d+2s}}(w^+(x)-w^+(y))\,dx\,dy\\&=\int_{\mathbb{R}^d}\int_{\mathbb{R}^d}J(x,y)\frac{(w^+(x)-w^+(y))^2+w^-(x)w^+(y)+w^+(x)w^-(y)}{|x-y|^{d+2s}}\,dx\,dy\\&\geq\int_{\mathbb{R}^d}\int_{\mathbb{R}^d}J(x,y)\frac{(w^+(x)-w^+(y))^2}{|x-y|^{d+2s}}\,dx\,dy>0\end{align*} if $w^+\neq0$, since $w^-(x)w^+(x)=w^-(y)w^+(y)=0$.
\end{proof}

\begin{remark}
With exactly the same argument by replacing $w^+$ with $w=u-v$, the operator $\mathcal{L}_g^s$ is strictly monotone. This also follows directly from the fact that \eqref{gGrowthCond} implies the strict monotonicity of $g$ (see for instance, page 2 of \cite{ChallalLyaghfouri2009CPAAHolderContinuityALaplaceEq}): 
for all $\xi,\zeta\in\mathbb{R}$ such that $\xi\neq\zeta$, \begin{equation}\label{gmonotone}(g_{:}(|\xi|)\xi-g_{:}(|\zeta|)\zeta)\cdot(\xi-\zeta)>0\quad \text{ a.e. }x,y\in\mathbb{R}^d.\end{equation} 

The strict monotonicity immediately implies the uniqueness of the solution in Proposition \ref{EllipticDirichletExistThm}.

\end{remark}

\begin{remark}\label{LpCoercive}
    In the particular case when $g(x,y,r)=|r|^{p-2}K(x,y)$ as in the fractional $p$-Laplacian \eqref{LpLap}, with $1<p<\infty$ and $K$ satisfies \eqref{KAssump}, the operator $\mathcal{L}_p^s$ is strictly coercive, in the sense that 
\begin{equation}\label{LgCoerciveEq}\langle \mathcal{L}_p^su- \mathcal{L}_p^sv,u-v\rangle
\geq 
\begin{cases} 2^{1-p}k_*[u-v]_{W^{s,p}_0(\Omega)}^p&\text{ if }p\geq2,\\
(p-1)2^{\frac{p^2-4p+2}{p}}k_*\dfrac{[u-v]_{W^{s,p}_0(\Omega)}^2}{\left([u]_{W^{s,p}_0(\Omega)}+[v]_{W^{s,p}_0(\Omega)}\right)^{2-p}}&\text{ if }1<p<2,\end{cases}\end{equation} where the seminorm of $W^{s,p}_0(\Omega)$ is given by 
\[[u]_{W^{s,p}(\Omega)}=\left(\int_{\mathbb{R}^d}\int_{\mathbb{R}^d}\frac{|u(x)-u(y)|^p}{|x-y|^{d+sp}}dxdy\right)^\frac{1}{p},\] This is a generalisation of Proposition 2.4 of \cite{LoRod2024StabilityCoincidenceSet} to the $K$-anisotropic case. 

\end{remark}


In the Hilbertian framework, we furthermore assume that $g(x,y,r)\in[\gamma_*,\gamma^*]$ as in \eqref{tildegcond}. Then, for a.e. $x,y\in\mathbb{R}^d$, it is easy to see from the proof of Theorem \ref{LgTMonotone} that for all $\xi,\zeta\in\mathbb{R}$, \[(g(x,y,|\xi|)\xi-g(x,y,|\zeta|)\zeta)\cdot(\xi-\zeta)\geq \gamma_*g_*|\xi-\zeta|^2\] and \[|g(x,y,|\xi|)\xi-g(x,y,|\zeta|)\zeta|\leq \gamma^*g^*|\xi-\zeta|.\]

\begin{proposition}\label{LgCoercive}
     The operator $\mathcal{L}_g^s$ in $H^s_0(\Omega)$ with $g(x,y,r)\in[\gamma_*,\gamma^*]$ satisfying \eqref{tildegcond} is strictly coercive and Lipschitz continuous. 
\end{proposition}

\begin{proof}
$\bar{\mathcal{L}}_g^s$ is strictly coercive for all $u,v\in H^s_0(\Omega)$ because
\begin{align*}\langle\bar{\mathcal{L}}_g^su-\bar{\mathcal{L}}_g^sv,u-v\rangle&=\int_{\mathbb{R}^d}\int_{\mathbb{R}^d}(g_{:}(|\delta^s u|)\delta^s u-g_{:}(|\delta^s v|)\delta^s v)\cdot (\delta^s u-\delta^s v)\frac{dx\,dy}{|x-y|^d}\\&\geq \gamma_*g_*\int_{\mathbb{R}^d}\int_{\mathbb{R}^d}|\delta^s u-\delta^s v|^2\frac{dx\,dy}{|x-y|^d}=\gamma_*g_*\norm{u-v}_{H^s_0(\Omega)}^2.\end{align*}

Also, $\mathcal{L}_g^s$ is Lipschitz since for all $u,v,w\in H^s_0(\Omega)$ with $\norm{w}_{H^s_0(\Omega)}=1$,
\begin{align*}|\langle \mathcal{L}_g^su-\mathcal{L}_g^sv,w\rangle|&\leq\int_{\mathbb{R}^d}\int_{\mathbb{R}^d}\left|g(x,y,|\delta^s u|)\delta^s u-g(x,y,|\delta^s v|)\delta^s v\right||\delta^s w|\frac{dx\,dy}{|x-y|^d}\\&\leq\gamma^*g^* \int_{\mathbb{R}^d}\int_{\mathbb{R}^d}\frac{|\delta^s u-\delta^s v|}{|x-y|^{\frac{d}{2}}}\frac{|w(x)-w(y)|}{|x-y|^{s+\frac{d}{2}}}\,dx\,dy\leq \gamma^*g^*\norm{u-v}_{H^s_0(\Omega)}.\end{align*}

\end{proof}

As a result, we have, in addition, the comparison property for the Dirichlet problem. 
Recall that we characterise an element $F\in[W^{-s,G_{:}^*}(\Omega)]^+$, the positive cone of the dual space of $W^{s,G_{:}}_0(\Omega)$, by
\begin{equation}\label{DualCone}F\geq0 \text{ in } W^{-s,G_{:}^*}(\Omega)\quad\quad\text{ if and only if }\quad\quad\langle F,v\rangle\geq0\quad\forall v\geq0,v\in W^{s,G_{:}}_0(\Omega).\end{equation}

\begin{proposition}\label{ComparisonPropDirichlet}
If $u,\hat{u}$ denotes the solution of \eqref{EllipticProbEq} corresponding to $F,\psi$ and $\hat{F},\hat{\psi}$ respectively, then 
\[F\geq\hat{F}\quad\text{ implies }\quad u\geq\hat{u}\quad \text{ a.e. in }\Omega.\]

\end{proposition}

\begin{proof}
Taking $v=u\vee\hat{u}$ for the original problem and $\hat{v}=u\wedge\hat{u}$ for the other problem and adding, we have 
\[\langle \mathcal{L}_g^s\hat{u}-\mathcal{L}_g^su,(\hat{u}-u)^+\rangle+ \langle F-\hat{F}, (\hat{u}-u)^+ \rangle =0.\] Since $F\geq\hat{F}$, the result follows by the strict T-monotonicity of $\mathcal{L}_g^s$.
\end{proof}

\begin{remark}
    This property of $\mathcal{L}_g^s$ extends and implies Lemma 9 of \cite{LindgrenLindqvist2014CVPDEFractionalEigenvalues} for the fractional $p$-Laplacian, as well as the fractional $g$-Laplacian in Proposition C.4 of \cite{BonderSalortVivas2020InteriorGlobalRegFracGLap} and Theorem 1.1 of \cite{MolinaSalortVivas2021NonAFracGLapProperties}.
\end{remark}

\begin{remark}
    This comparison property includes the result in Theorem 5.2 of \cite{BahrouniJMAA2018ComparisonPrincipleFracPxLap} in the case of a single non-homogeneous exponent $p(x,y)$ and it extends easily the validity of the sub-supersolutions principles to this more general class of operators $\mathcal{L}_g^s$.
\end{remark}
\subsection{Lewy-Stampacchia Inequalities for Obstacle Problems}

Next, we extend the comparison results for the obstacle problems
\begin{equation}\label{LgObsProb}u\in \mathbb{K}^s:\quad\langle\mathcal{L}_g^su-F,v-u\rangle\geq0\quad\forall v\in \mathbb{K}^s,\end{equation} for $F\in W^{-s,G_{:}^*}(\Omega)$ and measurable obstacle functions $\psi,\varphi\in W^{s,G_{:}}(\mathbb{R}^d)$ such that the closed convex sets $\mathbb{K}^s=\mathbb{K}^s_1$ or $\mathbb{K}^s_2$ defined by 
\[\mathbb{K}_1^s=\{v\in W^{s,G_{:}}_0(\Omega):v\geq\psi \text{ a.e. in }\Omega\}\neq\emptyset,\] 
\[\mathbb{K}_2^s=\{v\in W^{s,G_{:}}_0(\Omega):\psi\leq v\leq\varphi \text{ a.e. in }\Omega\}\neq\emptyset.\]

\begin{theorem}\label{EllipticObsProbThm}
The one or two obstacles problem \eqref{LgObsProb} has a unique solution $u=u(F,\psi,\varphi)\in \mathbb{K}^s$, respectively for $\mathbb{K}^s=\mathbb{K}^s_1$ or $\mathbb{K}^s_2$,and is equivalent to minimising in $\mathbb{K}^s$ the functional $\mathcal{G}_s$ defined in \eqref{LgFunc}.

Moreover, if $\hat{u}$ denotes the solution corresponding to $\hat{F}$, $\hat{\psi}$ or to $\hat{F}$, $\hat{\psi}$ and $\hat{\varphi}$, respectively, then 
\[F\geq\hat{F},\quad\psi\geq\hat{\psi}\quad\text{ implies }\quad u\geq\hat{u}\text{ a.e. in }\Omega,\] 
or
\[F\geq\hat{F},\quad\varphi\geq\hat{\varphi},\quad\psi\geq\hat{\psi}\quad\text{ implies }\quad u\geq\hat{u}\text{ a.e. in }\Omega,\] 
and if $F=\hat{F}$, the following $L^\infty$ estimates hold: \begin{equation}\label{LipContDepResult1}\norm{u-\hat{u}}_{L^\infty(\Omega)}\leq\|\psi-\hat{\psi}\|_{L^\infty(\Omega)}.\end{equation}
\begin{equation}\label{LipContDepResult2}\norm{u-\hat{u}}_{L^\infty(\Omega)}\leq\|\psi-\hat{\psi}\|_{L^\infty(\Omega)}\vee\norm{\varphi-\hat{\varphi}}_{L^\infty(\Omega)}.\end{equation}
\end{theorem}

\begin{proof}
The comparison property is once again standard and follows from the T-monotonicity of $\mathcal{L}_g^s$ as given in Theorem \ref{LgTMonotone}. Indeed, in both one or two obstacles, taking $v=u\vee\hat{u}\in \mathbb{K}^s$ in the problem \eqref{LgObsProb} for 
$u$ and $\hat{v}=u\wedge\hat{u}\in \hat{\mathbb{K}}^s$ in the problem \eqref{LgObsProb} for $\hat{u}$, by adding, we have 
\[\langle \mathcal{L}_g^s\hat{u}-\mathcal{L}_g^su,(\hat{u}-u)^+\rangle+ \langle F-\hat{F}, (\hat{u}-u)^+ \rangle\leq0.\] Since $F\geq\hat{F}$ and $\mathcal{L}_g^s$ is strictly T-monotone, $(\hat{u}-u)^+=0$, i.e. $u\geq\hat{u}$.

For the $L^\infty$-continuous dependence, the argument is similar, by taking, respectively, for the one or for the two obstacles problem $v=u+w\in \mathbb{K}^s$ and  $\hat{v}=\hat{u}-w\in \hat{\mathbb{K}}^s$ with $w=\left(\hat{u}-u-\|\psi-\hat{\psi}\|_{L^\infty(\Omega)}\right)^+$ or $w=\left(\hat{u}-u-\|\psi-\hat{\psi}\|_{L^\infty(\Omega)}\vee\norm{\varphi-\hat{\varphi}}_{L^\infty(\Omega)}\right)^+$.

The existence and uniqueness of the solution follow from well known results of convex analysis, since the functional $\mathcal{G}_s$ is strictly convex, lower semi-continuous and coercive, and $\mathbb{K}^s$ is a nonempty, closed convex set in both cases.

\end{proof}

Next, recall that the order dual of the space $W^{s,G_{:}}_0(\Omega)$, denoted by $W^{-s,G_{:}^*}_\prec(\Omega)$, is the space of finite energy measures \begin{equation}\label{OrderDualDef}W^{-s,G_{:}^*}_\prec(\Omega)=[W^{-s,G_{:}^*}(\Omega)]^+-[W^{-s,G_{:}^*}(\Omega)]^+,\end{equation} defined with the norm of $W^{-s,G_{:}^*}(\Omega)$, where $[W^{-s,G_{:}^*}(\Omega)]^+$ is the cone of positive finite energy measures in $W^{-s,G_{:}^*}(\Omega)$, as given in \eqref{DualCone}. 
Then, we have the following Lewy-Stampacchia inequalities.

\begin{theorem}\label{LewyStampacchia1+2}
Assume, in addition, that for the one or the two obstacles problem, respectively,
\[F,(\mathcal{L}_g^s\psi-F)^+\in W^{-s,G_{:}^*}_\prec(\Omega),\]
or
\[F,(\mathcal{L}_g^s\psi-F)^+,(\mathcal{L}_g^s\varphi-F)^+\in W^{-s,G_{:}^*}_\prec(\Omega).\] Then, the solution $u$ of the one or the two obstacles problem \eqref{LgObsProb}, satisfies in $W^{-s,G_{:}^*}(\Omega)$
\begin{equation}\label{ObsPbH-sEqLewyStamp}F\leq\mathcal{L}_g^su\leq F\vee\mathcal{L}_g^s\psi,\end{equation}
or
\begin{equation}\label{ObsPbH-sEqLewyStamp2}F\wedge\mathcal{L}_g^s\varphi\leq \mathcal{L}_g^su\leq F\vee\mathcal{L}_g^s\psi,\end{equation}
respectively. Consequently, in both cases $\mathcal{L}_g^su\in W^{-s,G_{:}^*}_\prec(\Omega)$.
\end{theorem}

\begin{proof}
    Since the operator $\mathcal{L}_g^s$ is strictly T-monotone, we can apply the abstract results of \cite[Theorem 2.4.1]{Mosco1976Lectures} and \cite[Theorem 4.2]{RodriguesTeymurazyan} for the one-obstacle and two-obstacles problems respectively. 

    Finally, the regularity of $\mathcal{L}_g^su$ follows from the fact that intervals are closed in order duals.
\end{proof}
\begin{remark}
    In fact, the results in Theorem 2.4.1 of \cite{Mosco1976Lectures} and in and \cite[Theorem 4.2]{RodriguesTeymurazyan} do not even require $\mathcal{G}_s$ in \eqref{LgFuncPotential} to be a potential operator, but only the strict T-monotonicity and the coercivity.

    For the one obstacle problem, since the associated functional $\mathcal{G}_s$ in \eqref{LgFuncPotential} is a potential operator which is submodular, as a consequence of T-monotonicity (see also Sections 3.1, 3.2 and 4.1 of \cite{AttouchPic_NPT_AFSToul_1979}), the Lewy-Stampacchia inequalities in the order dual $W^{-s,G_{:}^*}_\prec(\Omega)$ are also a consequence of Theorem 2.4 of \cite{GMosconi}.
\end{remark}
In particular, since $L^{\hat{G}^*_\cdot}(\Omega)\subset W^{-s,G_{:}^*}_\prec(\Omega)$, we have

\begin{corollary}\label{CorallaryLewyStamp}
The solution $u$ to the one or two obstacles problem \eqref{LgObsProb}  is also such that $\mathcal{L}_g^su\in L^{\hat{G}^*_\cdot}(\Omega)=[L^{\tilde{G}_\cdot}(\Omega)]^*$, provided we assume the stronger assumption
 \[f,(\mathcal{L}_g^s\psi-f)^+\in L^{\hat{G}^*_\cdot}(\Omega),\] 
 or
  \[f,(\mathcal{L}_g^s\psi-f)^+,(\mathcal{L}_g^s\varphi-f)^+\in L^{\hat{G}^*_\cdot}(\Omega),\]
 as then the Lewy-Stampacchia inequalities hold pointwise almost everywhere
 \begin{equation}f\leq \mathcal{L}_g^su\leq f\vee\mathcal{L}_g^s\psi \quad\text{ a.e. in }\Omega.\end{equation}
 or
\begin{equation}f\wedge\mathcal{L}_g^s\varphi\leq \mathcal{L}_g^su\leq f\vee\mathcal{L}_g^s\psi \quad\text{ a.e. in }\Omega.\end{equation}
\end{corollary}

\begin{proof}
    This follows simply by recalling that $W^{s,G_{:}}_0(\Omega)$ is dense in $L^{\tilde{G}_\cdot}(\Omega)$, and therefore the Lewy-Stampacchia inequalities taken in the dual space $W^{-s,G_{:}^*}(\Omega)$ reduce to integrals, as in \eqref{IntegralForm}, and it follows then that they hold also a.e. in $\Omega$.
\end{proof}


\section{Approximation by Semilinear Problems and Regularity}

The order properties implied by the strict T-monotonicity, in the case of integrable data, also allow the approximation of the solutions to the obstacle problems via bounded penalisation, which provides a direct way to prove the preceding Corollary \ref{CorallaryLewyStamp}
 and to reduce the regularity of their solutions to the regularity in the fractional Dirichlet problem.

\subsection{Approximation via Bounded Penalisation}

When the data $f$ and $(\mathcal{L}_g^s\psi-f)^+$ are integrable functions, the a.e. Lewy-Stampacchia inequalities can be obtained directly by approximation with a classical bounded penalisation of the obstacles. In the fractional $p$-Laplacian case it is even possible to estimate the error in the $W^{s,p}_0(\Omega)$-norm \cite{LoRod2024StabilityCoincidenceSet}. 
We first begin with the following auxiliary convergence result, which is well-known in other classical monotone cases, and in the framework of the operator $\mathcal{L}_g^s$ is due to \cite[Theorem 3.17]{AlbuquerqueAssisCarvalhoSalort2023-FracMusielakSobolevSpaces}.

\begin{lemma}\label{ConvgLemma}
Under assumptions \eqref{gGrowthCond}, suppose $\{u_n\}_{n\in\mathbb{N}}$ is a sequence in $W^{s,G_{:}}_0(\Omega)$. Then $u_n\to u$ strongly in $W^{s,G_{:}}_0(\Omega)$ if and only if
\begin{equation}\label{ConvgLemmaAssump}\limsup_{n\to\infty}\langle\mathcal{L}_g^su_n-\mathcal{L}_g^su,u_n-u\rangle=0.\end{equation}
\end{lemma}

Consider the penalised problem with $f$ and $\zeta=(\mathcal{L}_g^s\psi-f)^+\in L^{\hat{G}^*_\cdot}(\Omega)$,
\begin{equation}\label{PenalProb}u_\varepsilon\in W^{s,G_{:}}_0(\Omega):\quad\langle\mathcal{L}_g^su_\varepsilon,v\rangle+\int_\Omega\zeta\theta_\varepsilon(u_\varepsilon-\psi)v=\int_\Omega (f+\zeta)v, \quad\forall v\in W^{s,G_{:}}_0(\Omega),\end{equation} 
where 
$\theta_\varepsilon(t)$ is an approximation to the multi-valued Heaviside graph defined by
\[\theta_\varepsilon(t)=\theta\left(\frac{t}{\varepsilon}\right),\quad t\in\mathbb{R}\] 
for any fixed nondecreasing Lipschitz function $\theta:\mathbb{R}\to[0,1]$ satisfying \[\theta\in C^{0,1}(\mathbb{R}),\quad\theta'\geq0,\quad\theta(+\infty)=1\quad\text{ and }\theta(t)=0\text{ for }t\leq0;\]\[\exists C_\theta>0:[1-\theta(t)]t\leq C_\theta,\quad t>0.\]  

Then we have a direct proof of the Lewy-Stampacchia inequalities.

\begin{theorem}\label{LewyStampacchiaNonlin} Assume that  
\[f,(\mathcal{L}_g^s\psi-f)^+\in L^{\hat{G}^*_\cdot}(\Omega).\] 
Then, the solution $u$ of the nonlinear one obstacle problem satisfies 
\begin{equation}\label{LewyStampacchiaNonlinIneq}
f\leq \mathcal{L}_g^su\leq f\vee \mathcal{L}_g^s\psi\quad\text{ a.e. in }\Omega.
\end{equation} 
In particular, $\mathcal{L}_g^su\in L^{\hat{G}^*_\cdot}(\Omega)$. 

Furthermore, we have that the solution $u_\varepsilon$ of the penalised problem \eqref{PenalProb} converges to $u$ in the following sense:
\begin{equation}
    u_\varepsilon \to u \text{ strongly in }W^{s,G_{:}}_0(\Omega), \quad \text{ and } \quad u_\varepsilon \to u \text{ strongly in }L^{q^*}(\Omega)
\end{equation} for $q^*$ satisfying $1\leq q^* < \frac{d(g_*+1)}{d-s(g_*+1)}$.
\end{theorem}

\begin{proof}
For the one obstacle problem, the proof follows as in the linear case, given in Theorem 4.6 of \cite{FracObsRiesz} with the second obstacle $\varphi=+\infty$.  
In the general case, there exists a unique solution $u_\varepsilon$ to \eqref{PenalProb} by Theorem \ref{EllipticDirichletExistThm}. Next, we show that $u_\varepsilon\geq\psi$, so that the solution $u_\varepsilon\in \mathbb{K}^s$ for each $\varepsilon>0$. 
Indeed, for all $v\in W^{s,G_{:}}_0(\Omega)$ such that $v\geq0$, we have \begin{equation}\label{ObsPbObsIneq}\langle\mathcal{L}_g^s\psi-f+f,v\rangle\leq\langle (\mathcal{L}_g^s\psi-f)^++f,v\rangle\leq\int_\Omega(\zeta+f)v.\end{equation} Taking $v=(\psi-u_\varepsilon)^+\geq0$ and subtracting \eqref{PenalProb} from the above equation, we have \begin{align*}&\quad\langle\mathcal{L}_g^s\psi,(\psi-u_\varepsilon)^+\rangle-\langle\mathcal{L}_g^su_\varepsilon,(\psi-u_\varepsilon)^+\rangle\\&\leq\int_\Omega(\zeta+f)(\psi-u_\varepsilon)^++\int_\Omega\zeta\theta_\varepsilon(u_\varepsilon-\psi)(\psi-u_\varepsilon)^+-\int_\Omega(f+\zeta)(\psi-u_\varepsilon)^+\\&=\int_\Omega\zeta\theta_\varepsilon(u_\varepsilon-\psi)(\psi-u_\varepsilon)^+\\&=0.\end{align*} The last equality is true because either $u_\varepsilon-\psi>0$ which gives $(\psi-u_\varepsilon)^+=0$, or $u_\varepsilon-\psi\leq0$ which gives $\theta_\varepsilon(u_\varepsilon-\psi)=0$ by the construction of $\theta$, thus implying $\theta_\varepsilon(u_\varepsilon-\psi)(\psi-u_\varepsilon)^+=0$. By the T-monotonicity of $\mathcal{L}_g^s$, $(\psi-u_\varepsilon)^+=0$, i.e. $u_\varepsilon\in \mathbb{K}^s$ for any $\varepsilon>0$. 

Then, we show that $u_\varepsilon\geq\psi$ converges strongly in $W^{s,G_{:}}_0(\Omega)$ as $\varepsilon\to0$ to some $u$, which by uniqueness, is the solution of the obstacle problem. Indeed, taking $v=w-u_\varepsilon$ in \eqref{PenalProb} for arbitrary $w\in\mathbb{K}^s$, we have \begin{align*}\langle \mathcal{L}_g^su_\varepsilon,w-u_\varepsilon\rangle&=\int_\Omega(f+\zeta)(w-u_\varepsilon)-\int_\Omega\zeta\theta_\varepsilon(u_\varepsilon-\psi)(w-u_\varepsilon)\\&=\int_\Omega f(w-u_\varepsilon)+\int_\Omega\zeta[1-\theta_\varepsilon(u_\varepsilon-\psi)](w-u_\varepsilon)\\&\geq\int_\Omega f(w-u_\varepsilon)+\int_\Omega\zeta[1-\theta_\varepsilon(u_\varepsilon-\psi)](\psi-u_\varepsilon)\\&=\int_\Omega f(w-u_\varepsilon)-\varepsilon\int_\Omega\zeta[1-\theta_\varepsilon(u_\varepsilon-\psi)]\frac{u_\varepsilon-\psi}{\varepsilon}\\&\geq\int_\Omega f(w-u_\varepsilon)-\varepsilon C_\theta\int_\Omega\zeta\end{align*} since $\zeta,1-\theta_\varepsilon,w-\psi\geq0$ for $w\in\mathbb{K}^s_\psi$. 

Now, taking $w=u$, we obtain \[\langle \mathcal{L}_g^su_\varepsilon -f,u-u_\varepsilon\rangle\geq -\varepsilon C_\theta\int_\Omega\zeta,\] and letting $v=u_\varepsilon\in\mathbb{K}^s_\psi$ in the original obstacle problem \eqref{LgObsProb}, we have \[\langle \mathcal{L}_g^su -f,u_\varepsilon-u\rangle\geq 0.\] 
Taking the difference of these two equations, we have 
\begin{equation}\label{BddPenalEst}\varepsilon C_\theta\int_\Omega\zeta\geq\langle\mathcal{L}_g^su_\varepsilon-\mathcal{L}_g^su,u_\varepsilon-u\rangle.\end{equation} Applying the previous lemma, we have that $u_\varepsilon\to u$ strongly in $W^{s,G_{:}}_0(\Omega)$ as $\varepsilon\to0$.

Then, choosing $\zeta=(\mathcal{L}_g^s\psi-f)^+$ in the penalised problem, the inequality \eqref{LewyStampacchiaNonlinIneq} is also satisfied for $u_\varepsilon$, and since $\mathcal{L}_g^s$ is monotone, \eqref{LewyStampacchiaNonlinIneq} is therefore satisfied weakly by $u$ at the limit $\varepsilon\to0$.

Finally, the $L^q(\Omega)$ strong convergence follows easily using the compactness result in Corollary \ref{CompactnessTheoremRellichLp}.

\end{proof}

\begin{remark}
    Similar results hold for the two obstacles problem, in particular we have

    \begin{equation}\label{LewyStampacchiaNonlinIneq2}
f\wedge\mathcal{L}_g^s\phi\leq \mathcal{L}_g^su\leq f\vee \mathcal{L}_g^s\psi\quad\text{ a.e. in }\Omega.
\end{equation} 
    
    Indeed, the two obstacles problem follows similarly using the bounded penalised problem 
    \[u_\varepsilon\in W^{s,G_{:}}_0(\Omega):\quad\langle \mathcal{L}_g^su_\varepsilon,v\rangle+\int_\Omega\zeta_\psi\theta_\varepsilon(u_\varepsilon-\psi)v-\int_\Omega\zeta_\varphi\theta_\varepsilon(\varphi-u_\varepsilon)v=\int_\Omega (f+\zeta_\psi-\zeta_\varphi)v\quad\forall v\in W^{s,G_{:}}_0(\Omega)\] for \[\zeta_\psi\geq(\mathcal{L}_g^s\psi-f)^+,\quad\zeta_\varphi\geq(\mathcal{L}_g^s\varphi-f)^-,\] with $\theta_\varepsilon(t)=1$ for $t\geq\varepsilon$, followed by taking to the limit of $u_\varepsilon$ to $u$ with the choice $\zeta_\psi=(\mathcal{L}_g^s\psi-f)^+$ and $\zeta_\varphi=(\mathcal{L}_g^s\varphi-f)^-$. 

\end{remark}

\begin{remark}
In the particular case when $g(x,y,r)=|r|^{p-2)} K(x,y)$ for $1<p<\infty$ and $\mathcal{L}_g^s$ corresponds to the fractional $p$-Laplacian, by Remark \ref{LpCoercive}, we furthermore have the estimate 
\[[u_\varepsilon-u]_{W^{s,p}_0(\Omega)}\leq  C_p \varepsilon^{1/(p\vee2)}\] for some constant $C_p$ depending on $p$, $\zeta_\psi$, $\zeta_\varphi$, $k_*$, $k^*$ and $f$.
In particular, this implies that $u_\varepsilon$ converges strongly in $W^{s,p}_0(\Omega)$ to $u$ as $\varepsilon\to0$. \cite{LoRod2024StabilityCoincidenceSet}
    
\end{remark}

\subsection{Regularity in Obstacle Problems}

As an immediate corollary of the approximation with the bounded penalisation, 
based on the regularity results for the Dirichlet problem in Section \ref{sec:DirichletPb}, we can extend these regularity results to the obstacle problems. The first is the uniform boundedness results of their solutions as a corollary of Theorems \ref{UnifBddThmMusielak}.

\begin{theorem}
Suppose $F=f$ and $f\vee\mathcal{L}_{p}^s\psi\in L^m(\Omega)$, with $m > \frac{d}{s(g_*+1)}$ and $g$ satisfies \eqref{gGrowthCond} with $s(g_*+1)<d$. Then, the solution $u$ of the one obstacle problem \eqref{LgObsProb} is bounded, i.e. $u\in L^\infty(\Omega)$. If, in addition, $f\wedge\mathcal{L}_{p}^s\varphi\in L^m(\Omega)$ the solution $u$ of the two obstacles problem also satisfies $u\in L^\infty(\Omega)$.

\end{theorem}

Next, we have the H\"older regularity results for the solution to the obstacle problem.

\begin{theorem}
Let $F=f\in L^\infty(\Omega)$. Suppose either
\begin{enumerate}[label=(\alph*)]
    \item $g(x,y,r)$ is of the form $|r|^{p(x,y)-2}K(x,y)$ as in the fractional $p$-Laplacian $\mathcal{L}^s_p$ in \eqref{LpLap} for $1<p_-\leq p(x,y)\leq p_+<\infty$, and $K$ satisfies \eqref{KAssump}, with $p(\cdot,\cdot)$ and $K(\cdot,\cdot)$ symmetric, such that $p(x,y)$ is log-H\"older continuous on the diagonal $D=\{(x,x):x\in\Omega\}$, i.e. 
\[\sup_{0<r\leq 1/2}\left[\log\left(\frac{1}{r}\right)\sup_{B_r\subset\Omega}\sup_{x_2,y_1,y_2\in B_r}|p(x_1,y_1)-p(x_2,y_2)|\right]\leq C \quad\text{ for some }C>0,\]

    \item $g$ is isotropic, i.e. $g=g(r)$ is independent of $(x,y)$, with $G$ satisfying the $\Delta'$ condition,

    \item $g$ is isotropic with $\bar{g}=\bar{g}(r)$ convex in $r$ and $g_*\geq1$ in \eqref{gGrowthCond}, or

    \item $g(x,y,r)$ is uniformly bounded and positive as in \eqref{tildegcond} with symmetric anisotropy
\end{enumerate}
If $f,f\vee\mathcal{L}_{p}^s\psi\in L^\infty(\Omega)$ in the one obstacle problem and also $f\wedge\mathcal{L}_{p}^s\varphi\in L^\infty(\Omega)$ in the two obstacles problem, their solutions $u$ are H\"older continuous, i.e., in cases (a) and (b), locally in $\Omega$, \[u\in C^\alpha(\Omega)\text{ for some }0<\alpha<1.\] 
and, in cases (c) and (d), up to the boundary, \[u\in C^\alpha(\bar{\Omega})\text{ for some }0<\alpha<1.\]

\end{theorem}

\begin{remark}
The result for (a) was previously given for the isotropic fractional $p$-Laplacian for $\psi$ H\"older continuous in Theorem 6 of \cite{K+Kuusi+Palatucci2016CVPDE-ObstaclePbFracPLap} or Theorem 1.3 of \cite{Piccinini2022NonA-ObstaclePbFracPLap}.
\end{remark}

\begin{remark}
 In the case when $g(x,y,r)$ is uniformly bounded and positive as in \eqref{tildegcond}, if $f, f\vee\mathcal{L}_{p}^s\psi\in L^q_{loc}(\Omega)$ (and $f\wedge\mathcal{L}_{p}^s\varphi\in L^q_{loc}(\Omega)$, resp. for the two obstacles problem) for some $q>\frac{2d}{d+2}$, then, the solutions $u$ of the obstacle problems are such that  $u\in W^{s+\delta,2+\delta}_{loc}(\Omega)$, for some positive $0<\delta<1-s$, by Theorem 1.1 of \cite{KuusiMingioneSire2015APDE-RegularityNonlocal} as stated in Part (a) of Theorem \ref{EllipticDirichletRegThm2}.
\end{remark}

\section{Capacities}

In this section, we make a brief introduction to the basic relation between the obstacle problem and potential theory, extending the seminal idea of Stampacchia \cite{StampacchiaEllipticEqsLmReg} to the fractional generalised Orlicz framework. Other nonlinear extensions to nonlinear potential theory have been considered by \cite{AttouchPic_NPT_AFSToul_1979}, for general Banach-Dirichlet spaces, by \cite{CapacityBook}, for weighted Sobolev spaces for $p$-Laplacian operators, and more recently by \cite{OrliczCapacities} in generalised Orlicz spaces for classical derivatives with a slightly different definition of capacity. 

\subsection{The Fractional Generalised Orlicz Capacity}

For $E\subset\Omega$, one says that $u\succeq0$ on $E$ (or $u\geq0$ on $E$ in the sense of $W^{s,G_{:}}_0(\Omega)$) if there exists a sequence of Lipschitz functions with compact support in $\Omega$ $u_k\to u$ in $W^{s,G_{:}}_0(\Omega)$ such that $u_k\geq0$ on $E$. Clearly if $u\succeq0$ on $E$, then also $u\geq0$ a.e. on $E$. On the other hand if $u\geq0$ a.e. on $\Omega$, then $u\succeq0$ on $\Omega$ (see for instance Proposition 5.2 of \cite{KinderlehrerStampacchia})

Let $E\subset\Omega$ be any compact subset. Define the nonempty closed convex set of $W^{s,G_{:}}_0(\Omega)$ by \[\mathbb{K}^s_E=\{v\in W^{s,G_{:}}_0(\Omega):v\succeq1\text{ on }E\},\] and consider the following variational inequality of obstacle type \begin{equation}\label{CapVarIneq}u\in \mathbb{K}^s_E:\langle\mathcal{L}_g^su,v-u\rangle\geq0,\quad\forall v\in \mathbb{K}^s_E.\end{equation} This variational inequality clearly has a unique solution and consequently we can also extend to the fractional generalised Orlicz framework the following theorem, which is due to Stampacchia \cite{StampacchiaEllipticEqsLmReg} for general linear second order elliptic differential operators with discontinuous coefficients. 

\begin{theorem}\label{RadonMeasure}
For any compact $E\subset\Omega$, the unique solution $u$ of \eqref{CapVarIneq}, called the \emph{$(s,G_{:})$-capacitary potential} of $E$, is such that \[u=1\text{ on }E\text{ (in the sense of }W^{s,G_{:}}_0(\Omega))\]\[\mu_{s,G_{:}}=\mathcal{L}_g^su\geq0\text{ with }supp(\mu_{s,G_{:}})\subset E.\]  

Moreover, for the non-negative Radon measure $\mu_{s,G_{:}}$, one has 
\begin{equation}\label{Lg-Capacity}
C_s^g(E)=\langle\mathcal{L}_g^su,u\rangle=\int_\Omega d\mu_{s,G_{:}}=\mu_{s,G_{:}}(E)
\end{equation}
and this number is the \emph{$(s,G_{:})$-capacity} of $E$ with respect to the operator $\mathcal{L}_g^su$.
\end{theorem}

\begin{proof}
The proof follows a similar approach to the classical case (\cite[Theorem 3.9]{StampacchiaEllipticEqsLmReg}  or \cite[Theorem 8.1]{ObstacleProblems}).
Taking $v=u\wedge1=u-(u-1)^+\in \mathbb{K}^s_E$ in \eqref{CapVarIneq}, one has, by T-monotonicity (Theorem \ref{LgTMonotone}), \[0<\langle\mathcal{L}_g^s(u-1),(u-1)^+\rangle=\langle\mathcal{L}_g^su,(u-1)^+\rangle\leq0\] since the $\delta^s$ is invariant for translations. Hence $u\geq 1$ in $\Omega$, which implies $u\preceq1$ in $\Omega$. But $u\in \mathbb{K}^s_E$, so $u\succeq1$ on $E$. Therefore, the first result $u=1$ on $E$ follows.

For the second result, set $v=u+\varphi\in \mathbb{K}^s_E$ in \eqref{CapVarIneq} with an arbitrary $\varphi\in C_c^\infty(\Omega)$, $\varphi\geq0$. Then, by the Riesz-Schwartz theorem (see for instance \cite[Theorem 1.1.3]{AdamsHedberg}), there exists a non-negative Radon measure $\mu_{s,G_{:}}$ on $\Omega$ such that \[\langle\mathcal{L}_g^su,\varphi\rangle=\int_\Omega\varphi \,d\mu_{s,G_{:}},\quad\forall\varphi\in C_c^\infty(\Omega).\]
Moreover, for $x\in\Omega\backslash E$, there is a neighbourhood $O\subset\Omega\backslash E$ of $x$ so that $u+\varphi\in \mathbb{K}^s_E$ for any $\varphi\in C_c^\infty(O)$. Therefore,\[\langle\mathcal{L}_g^su,\varphi\rangle=0,\quad\forall\varphi\in C_c^\infty(\Omega\backslash E)\]which means $\mu_{s,G_{:}}=\mathcal{L}_g^su=0$ in $\Omega\backslash E$. Therefore, $supp(\mu_{s,G_{:}})\subset E$ and the third result follows immediately.
\end{proof}

\begin{remark}
    In fact, the $(s,G_{:})$-capacity is a capacity of $E$ with respect to $\Omega$ in the same line as in Stampacchia \cite{StampacchiaEllipticEqsLmReg} (see also \cite{KinderlehrerStampacchia} and \cite{ObstacleProblems}).
    This type of characterisation of capacitary potentials and their relation to positive measures with finite energy have been also considered in an abstract nonlinear framework 
 in Banach-Dirichlet spaces, including classical Sobolev spaces, in\cite{AttouchPic_NPT_AFSToul_1979}.
\end{remark}

\begin{remark}
    For any subset $F\subset\Omega$, defining the capacity of $F$ by taking the supremum of the capacity for all compact sets $E\subset F$, it follows that the $(s,G_{:})$-capacity is an increasing set function and it is expected that it is a Choquet capacity, as in other general theories of linear and nonlinear potentials. For instance, see \cite{Warma2015FracCapacity} for the case of the linear operators in \eqref{La}, or in the case of the fractional $p$-Laplacian as in \eqref{LpLap} Theorem 1.1 of \cite{ShiZhang2020MathNachrFracPLapCapacity} and Theorem 2.4 of \cite{ShiXiao2016PotAnalFracPLapCapacity}, or a non-variational case in Theorem 4.1 of \cite{ShiXiao2017CVPDEFracCapPLap}. However, it is out of the scope of this work to pursue the theory of generalised Orlicz fractional capacity.
\end{remark}

\subsection{The $s$-Capacity in the $H^s_0(\Omega)$ Hilbertian Nonlinear Framework}\label{sect:HilbertianCapacity}

We are now particularly interested in extending Stampacchia's theory to the nonlinear Hilbertian framework associated with $\mathcal{L}^s_g$ for strictly positive and bounded $g$ satisfying \eqref{tildegcond}.

We denote by $C_s$ the capacity associated to the norm of $H^s_0(\Omega)$, which is defined for any compact set $E\subset\Omega$ by 
\[C_s(E)=\inf\left\{\norm{v}_{H^s_0(\Omega)}^2:v\in H^s_0(\Omega),v\succeq1 \text{ on }E\right\}=\langle(-\Delta)^s\bar{u},\bar{u}\rangle,\]
where $\bar{u}$ is the corresponding \emph{$s$-capacitary potential} of $E$.


We notice that the $C_s$-capacity corresponds to the capacity associated with the fractional Laplacian $(-\Delta)^s$ and the $s$-capacitary potential of a compact set $E$ is the solution of the obstacle problem \eqref{CapVarIneq} when $\mathcal{L}_g^s=(-\Delta)^s$ and the bilinear form \eqref{La} is the inner product in $H^s_0(\Omega)$. 

It is well-known (see for instance Theorem 5.1 of \cite{FracObsRiesz}) that for every $u\in H^s_0(\Omega)$, there exists a unique (up to a set of capacity 0) quasi-continuous function $\bar{u}:\Omega\to\mathbb{R}$ such that $\bar{u}=u$ a.e. on $\Omega$. Thus, it makes sense to identify a function $u\in H^s_0(\Omega)$ with the class of quasi-continuous functions that are equivalent quasi-everywhere (q.e.). Denote the space of such equivalent classes by $Q_s(\Omega)$. Then, for every element $u\in H^s_0(\Omega)$, there is an associated $\bar{u}\in Q_s(\Omega)$.

Define the space $L^2_{C_s}(\Omega)$ by 
\[L^2_{C_s}(\Omega)=\{\phi\in Q_s(\Omega):\exists u\in H^s_0(\Omega):\bar{u}\geq|\phi|\text{ q.e. in }\Omega\}\] and its associated norm (see \cite{AttouchPicard})
\[\norm{\phi}_{L^2_{C_s}(\Omega)}=\inf\{\norm{u}_{H^s_0(\Omega)}:u\in H^s_0(\Omega),\bar{u}\geq|\phi|\text{ q.e. in }\Omega\}.\]  
Then, $L^2_{C_s}(\Omega)$ is a Banach space and its dual space can be identified with the order dual of $H^s_0(\Omega)$ (by Theorem 5.6 of \cite{FracObsRiesz}), i.e. \[[L^2_{C_s}(\Omega)]'=H^{-s}(\Omega)\cap M(\Omega)=H^{-s}_\prec(\Omega)=[H^{-s}(\Omega)]^+-[H^{-s}(\Omega)]^+,\] where $M(\Omega)$ is the set of bounded measures in $\Omega$. Furthermore, by Proposition 5.2 of \cite{FracObsRiesz}, the injection of $H^s_0(\Omega)\cap C_c(\Omega)\hookrightarrow L^2_{C_s}(\Omega)$ is dense.

Now we consider the special Hilbertian case of Theorem \ref{RadonMeasure} for a nonlinear operator $\mathcal{L}_g^s$ when $g(x,y,r)$ corresponds to the nonlinear kernel under the assumptions \eqref{gGrowthCond} and \eqref{tildegcond}, i.e. such that $0<\gamma_*\leq g(x,y,r)\leq \gamma^*$ for $0<\gamma_*<1<\gamma^*$.
In this case, we have a simple comparison of the capacities.

\begin{theorem}\label{CapProp5.4}
For any subset $F\subset\Omega$, \[\gamma_*C_s(F)\leq C_s^g(F)\leq \frac{{\gamma^*}^2}{\gamma_*}C_s(F).\]
\end{theorem}
\begin{proof}
We first show it for a compact set $E\subset\Omega$.
Let $u$ be the $(s,G_{:})$-capacitary potential of $E$, and $\bar{u}$ be the $s$-capacitary potential of $E$. Since $\bar{u}\succeq1$ on $E$, we can choose $v=\bar{u}\in \mathbb{K}^s_{E}$ in \eqref{CapVarIneq} to get
\begin{align*}C_s^g(E)&=\langle\mathcal{L}_g^su,u\rangle\leq\langle\mathcal{L}_g^su,\bar{u}\rangle\\&\leq \gamma^*\norm{u}_{H^s_0(\Omega)}\norm{\bar{u}}_{H^s_0(\Omega)}\leq \frac{\gamma_*}{2}\norm{u}_{H^s_0(\Omega)}^2+\frac{{\gamma^*}^2}{2\gamma_*}\norm{\bar{u}}_{H^s_0(\Omega)}^2\\&\leq\frac{1}{2} \langle\mathcal{L}_g^su,u\rangle+\frac{{\gamma^*}^2}{2\gamma_*}C_s(E)=\frac{1}{2}C_s^g(E)+\frac{{\gamma^*}^2}{2\gamma_*}C_s(E)\end{align*} by Cauchy-Schwarz inequality and the coercivity of $g$. Similarly, we can choose $v=u\in \mathbb{K}^s_{E}$ for \eqref{CapVarIneq} for $C_s(E)$, with $\mathcal{L}_g^s=(-\Delta)^s$, using again the coercivity of $g$, and obtain
\begin{align*}C_s(E)=&\langle(-\Delta)^s\bar{u},\bar{u}\rangle\leq\langle(-\Delta)^s\bar{u},u\rangle\\&\leq\norm{\bar{u}}_{H^s_0(\Omega)}\norm{u}_{H^s_0(\Omega)}\leq \frac{1}{2}\norm{\bar{u}}_{H^s_0(\Omega)}^2+\frac{1}{2}\norm{u}_{H^s_0(\Omega)}^2\\&\leq \frac{1}{2}C_s(E)+\frac{1}{2\gamma_*}\langle\mathcal{L}_g^su,u\rangle= \frac{1}{2}C_s(E)+\frac{1}{2\gamma_*}C_s^g(E).\end{align*}

Finally, we can extend this result for general sets $F\subset\Omega$ by taking the supremum over all compact sets $E$ in $F$.
\end{proof}



As a simple application, we consider the corresponding nonlinear nonlocal obstacle problem in $L^2_{C_s}(\Omega)$. This extends some results of \cite{StampacchiaEllipticEqsLmReg} and \cite{AdamsCapacityObstacle} (see also \cite{ObstacleProblems}). See also Propositions 4.18 and 5.1 of \cite{AttouchPic_NPT_AFSToul_1979}, which gives the existence result in the local classical case of $W^{1,p}_0(\Omega)$.

\begin{theorem}
Let $\psi$ be an arbitrary function in $L^2_{C_s}(\Omega)$. Suppose that the closed convex set $\bar{\mathbb{K}}^s$ is such that \[\bar{\mathbb{K}}^s=\{v\in H^s_0(\Omega):\bar{v}\geq\psi \text{ q.e. in }\Omega\}\neq\emptyset.\] Then there is a unique solution to \begin{equation}\label{ComparisonCap1}u\in\bar{\mathbb{K}}^s:\quad\langle\mathcal{L}_g^su,v-u\rangle\geq0,\quad\forall v\in\bar{\mathbb{K}}^s,\end{equation} which is non-negative and such that \begin{equation}\label{ComparisonCap2}\norm{u}_{H^s_0(\Omega)}\leq(\gamma^*/\gamma_*)\norm{\psi^+}_{L^2_{C_s}(\Omega)}.\end{equation} Moreover, there is a unique measure $\mu_{s,g}=\mathcal{L}_g^su\geq0$, concentrated on the coincidence set $\{u=\psi\}=\{u=\psi^+\}$, verifying \begin{equation}\label{ComparisonCap3}\langle\mathcal{L}_g^su,v\rangle=\int_\Omega\bar{v}\,d\mu_{s,g},\quad\forall v\in H^s_0(\Omega),\end{equation} and \begin{equation}\label{ComparisonCap4}\mu_{s,g}(E)\leq\left(\frac{\gamma^{*2}}{\gamma_*^{3/2}}\right)\norm{\psi^+}_{L^2_{C_s}(\Omega)}\left[C_s^g(E)\right]^{1/2},\quad\forall E\Subset\Omega,\end{equation} in particular $\mu_{s,g}$ does not charge on sets of capacity zero.
\end{theorem}

\begin{proof}
By the maximum principle given in Theorem \ref{EllipticObsProbThm}, taking $v=u+u^-$, the solution is non-negative. Hence, the variational inequality \eqref{ComparisonCap1} is equivalent to solving the variational inequality with $\bar{\mathbb{K}}^s=\bar{\mathbb{K}}^s_\psi$ replaced by $\bar{\mathbb{K}}_{\psi^+}^s$. Since $\psi^+\in L^2_{C_s}(\Omega)$, by definition, $\bar{\mathbb{K}}_{\psi^+}^s\neq\emptyset$ and we can apply the Stampacchia theorem to obtain a unique non-negative solution. From \eqref{ComparisonCap1} it follows
\[\gamma_*\norm{u}_{H^s_0(\Omega)}^2\leq\langle\mathcal{L}_g^su,u\rangle\leq\langle\mathcal{L}_g^su,v\rangle\leq \gamma^*\norm{u}_{H^s_0(\Omega)}\norm{v}_{H^s_0(\Omega)},\]
and we have \[\norm{u}_{H^s_0(\Omega)}\leq(\gamma^*/\gamma_*)\norm{v}_{H^s_0(\Omega)},\quad\forall v\in \bar{\mathbb{K}}_{\psi^+}^s,\] giving \eqref{ComparisonCap2}, by using the definition of the $L^2_{C_s}(\Omega)$-norm of $\psi^+$.

The existence of a Radon measure for \eqref{ComparisonCap3} follows exactly as in Theorem \ref{RadonMeasure}. Finally, recalling the definitions, it is sufficient to prove \eqref{ComparisonCap4} for any compact subset $E\subset\Omega$. But this follows from \[\mu_{s,g}(E)\leq\int_\Omega\bar{v}\,d\mu_{s,g}=\langle\mathcal{L}_g^su,v\rangle\leq \gamma^*\norm{u}_{H^s_0(\Omega)}\norm{v}_{H^s_0(\Omega)}\leq\frac{\gamma^{*2}}{\gamma_*}\norm{\psi^+}_{L^2_{C_s}(\Omega)}\norm{v}_{H^s_0(\Omega)},\quad\forall v\in \mathbb{K}_E^s.\] Now, recall from Proposition \ref{CapProp5.4} that we have \[C_s^g(E)\geq \gamma_*C_s(E)=\gamma_*\inf_{v\in \mathbb{K}_E^s}\norm{v}_{H^s_0(\Omega)}^2\] thereby obtaining \eqref{ComparisonCap4}.
\end{proof}

\begin{corollary}
If $u$ and $\hat{u}$ are the solutions to \eqref{ComparisonCap1} with non-negative compatible obstacles $\psi$ and $\hat{\psi}$ in $L^2_{C_s}(\Omega)$ respectively, then \[\norm{u-\hat{u}}_{H^s_0(\Omega)}\leq k\|\psi-\hat{\psi}\|_{L^2_{C_s}(\Omega)}^{1/2},\] where $k=(\gamma^*/\gamma_*)\left[\norm{\psi}_{L^2_{C_s}(\Omega)}+\|\hat{\psi}\|_{L^2_{C_s}(\Omega)}\right]^{1/2}.$
\end{corollary}

\begin{proof}
Since $supp(\mu_{s,g})\subset\{u=\psi\}$ and $supp(\hat{\mu}_{s,g})\subset\{\hat{u}=\hat{\psi}\}$ (where $\mu_{s,g}=\mathcal{L}_g^su$ and $\hat{\mu}_{s,g}=\mathcal{L}_g^s\hat{u}$), for an arbitrary $w\in\bar{\mathbb{K}}^s_{|\psi-\hat{\psi}|}$, by setting $v=u-\hat{u}$ in \eqref{ComparisonCap3} for $\mu_{s,g}$ and for $\hat{\mu}_{s,g}$, we have 

\begin{align*}\gamma_*\norm{u-\hat{u}}_{H^s_0(\Omega)}^2&\leq\langle\mathcal{L}_g^su-\hat{u},u-\hat{u}\rangle=\langle\mathcal{L}_g^su,u-\hat{u}\rangle-\langle\mathcal{L}_g^s\hat{u},u-\hat{u}\rangle\\&=\int_\Omega(u-\hat{u})\,d\mu_{s,g}-\int_\Omega(u-\hat{u})\,d\hat{\mu}_{s,g}\leq\int_\Omega(\psi-\hat{\psi})\,d\mu_{s,g}-\int_\Omega(\psi-\hat{\psi})\,d\hat{\mu}_{s,g}\\&\leq\int_\Omega|\psi-\hat{\psi}|\,d(\mu_{s,g}+\hat{\mu}_{s,g})\leq\int_\Omega w\,d(\mu_{s,g}+\hat{\mu}_{s,g}) \\&=\int_\Omega w\,d\mu_{s,g}+\int_\Omega w\, d\hat{\mu}_{s,g}=\langle\mathcal{L}_g^su,w\rangle+\langle\mathcal{L}_g^s\hat{u},w\rangle\\&\leq \gamma^*\left[\norm{u}_{H^s_0(\Omega)}+\norm{\hat{u}}_{H^s_0(\Omega)}\right]\norm{w}_{H^s_0(\Omega)}\\&\leq\frac{\gamma^{*2}}{\gamma_*}\left[\norm{\psi}_{L^2_{C_s}(\Omega)}+\|\hat{\psi}\|_{L^2_{C_s}(\Omega)}\right]\norm{w}_{H^s_0(\Omega)}\text{ by \eqref{ComparisonCap2}}.\
\end{align*}

Since $w$ is arbitrary in $\bar{\mathbb{K}}^s_{|\psi-\hat{\psi}|}$, the conclusion follows by the definition of the norm of $|\psi-\hat{\psi}|$ in $L^2_{C_s}(\Omega)$.
\end{proof}

\bibliographystyle{plain}
\bibliography{ref.bib}

\end{document}